\documentclass[12pt]{amsart}
\usepackage[margin=1.25in]{geometry}
\usepackage[utf8]{inputenc}
\usepackage[T1]{fontenc}
\usepackage{amsmath}
\usepackage{amssymb}
\usepackage{amsthm}
\usepackage{hyperref}
\usepackage{array}
\usepackage{longtable}
\theoremstyle{plain}
\newtheorem{theorem}{Theorem}[section]
\newtheorem{proposition}[theorem]{Proposition}
\newtheorem{lemma}[theorem]{Lemma}
\newtheorem{corollary}[theorem]{Corollary}

\theoremstyle{definition} 
\newtheorem{definition}[theorem]{Definition}

\theoremstyle{remark} 
\newtheorem{remark}[theorem]{Remark}

\newcommand{\ku}{{\Bbbk}} 
\DeclareMathOperator{\VecCat}{Vec}
\newcommand{\ModCat}{\mathcal{M}}
\newcommand{\Vect}[1]{\VecCat_{#1}}

\newcommand{\cC}{\mathcal{C}}
\newcommand{\cB}{\mathcal{B}}
\newcommand{\cD}{\mathcal{D}}
\newcommand{\cE}{\mathcal{E}}
\newcommand{\cM}{\mathcal{M}}
\newcommand{\cN}{\mathcal{N}}
\newcommand{\TY}{\operatorname{TY}}
\newcommand{\Aut}{\operatorname{Aut}}
\newcommand{\AutT}{\operatorname{Aut}_{\otimes}} 

\newcommand{\Fun}{\operatorname{Fun}}
\newcommand{\FPdim}{\operatorname{FPdim}}
\newcommand{\id}{\operatorname{id}}
\newcommand{\Inv}{\operatorname{Inv}}
\newcommand{\Irr}{\operatorname{Irr}}

\DeclareMathOperator{\Ind}{Ind}
\DeclareMathOperator{\Res}{Res}
\newcommand{\Rep}{\operatorname{Rep}}
\newcommand{\St}{\operatorname{St}}
\newcommand{\Hom}{\operatorname{Hom}}
\newcommand{\Ad}{\operatorname{Ad}}

\newcommand{\Mod}{\operatorname{Mod}}
\newcommand{\Alt}{\operatorname{Alt}}
\newcommand{\pr}{\operatorname{pr}}
\title[Fiber Functors of Equivariantizations]{Fiber Functors of Equivariantizations of Finite Tensor Categories}

\author[C. Galindo]{C\'esar Galindo}
\address{Departamento de Matem\'aticas, Universidad de los Andes, Bogot\'a, Colombia}
\email{cn.galindo1116@uniandes.edu.co}

\author[C. Gallego]{Claudia Gallego}
\address{Departamento de Matem\'aticas, Pontificia Universidad Javeriana, Bogot\'a, Colombia}
\email{gallegoj.cm@javeriana.edu.co}

\author[Y. Morales]{Yiby Morales}
\address{School of Engineering and Science, Tecnol\'ogico de Monterrey, Quer\'etaro, M\'exico}
\email{yk.morales@tec.mx}

\date{}


\hypersetup{
  unicode=true,
  pdflang={en},
  hidelinks,
  pdftitle={Fiber Functors of Equivariantizations of Finite Tensor Categories},
  pdfauthor={César Galindo, Claudia Gallego, Yiby Morales},
  pdfsubject={Fiber functors on equivariantizations and gaugings},
  pdfkeywords={finite tensor category, equivariantization, fiber functor, module category, gauging}
}

\begin{document}

\begin{abstract}
Let $G$ be a finite group acting on a finite tensor category $\cC$. We classify fiber functors on the equivariantization $\cC^G$ in terms of equivariant exact module categories over $\cC$, indexed by subgroups of $G$. The data are a subgroup $H\subseteq G$ and an $H$-equivariant $\cC$-module category $\cM$ whose underlying $\cC$-module category is indecomposable, exact, and semisimple; they give a fiber functor precisely when $H$ acts transitively on the simple objects of $\cM$ and the stabilizer cocycle of one, hence every, simple object is non-degenerate. Through Tannaka--Krein reconstruction this describes realizations of $\cC^G$ as the representation category of a finite-dimensional Hopf algebra, with no semisimplicity hypothesis on $\cC$. As applications, for odd primes $p$ we determine the fiber functors on $\Rep(H_p)$, where $H_p$ denotes Nikshych's semisimple Hopf algebra of dimension $4p^2$~\cite{Nik08}: there is one equivalence class if $p\equiv3\pmod4$ and two if $p\equiv1\pmod4$. We also use the classification for gaugings to determine which non-pointed entries in the small-dimensional list of~\cite{GNradical} are representation categories of semisimple factorizable Hopf algebras.
\end{abstract}
\maketitle

\section{Introduction}

Let $\cC$ be a finite tensor category over an algebraically closed field $\ku$ of characteristic zero. A \emph{fiber functor} on $\cC$ is a $\ku$-linear exact faithful tensor functor $F:\cC\to\VecCat_{\ku}$. Classical neutral Tannakian duality reconstructs affine group schemes from symmetric tensor categories equipped with a fiber functor~\cite{Saavedra72, DM82}. In the non-symmetric setting relevant here, Tannaka--Krein reconstruction instead produces a Hopf algebra~\cite{JS91, Schauenburg92, BV07}: for every fiber functor $F$ on a finite tensor category $\cC$, there is a finite-dimensional Hopf algebra $A_F$ with $\cC\simeq\Rep(A_F)$, under which $F$ becomes the forgetful functor~\cite[Theorem~5.3.12]{EGNO15}. The fiber-functor problem is thus the problem of \emph{Hopf realization}: deciding when a finite tensor category is the representation category of a finite-dimensional Hopf algebra, and counting the realizations arising from inequivalent forgetful functors.

For group-theoretical fusion categories this problem has a concrete answer in terms of finite groups and cohomology. Recall that a fusion category is called \emph{group-theoretical} if it is categorically Morita equivalent to $\VecCat_G^\omega$, the category of finite-dimensional $G$-graded vector spaces with associativity constraint defined by a $3$-cocycle $\omega\in Z^3(G,\ku^\times)$. In this case fiber functors are described by subgroup and cochain data~\cite{Os03, Nik08}. Equivariantizations give a different source of examples. If a finite group $G$ acts on a tensor category $\cC$, then $\cC^G$ is the category of objects of $\cC$ equipped with compatible $G$-equivariant structures, and is again a finite tensor category. Equivariantizations of Tambara--Yamagami categories, the fusion categories with simple objects $A\cup\{m\}$ and fusion rule $m\otimes m=\bigoplus_{a\in A}a$, give non-group-theoretical semisimple Hopf algebras~\cite{Nik08, TY98}, answering~\cite[Question~8.45]{ENO05}. Thus non-group-theoretical fusion categories may still admit fiber functors; consequently, the group-theoretical classification does not cover the general fiber-functor problem. Our first result classifies fiber functors on an arbitrary equivariantization.

\begin{theorem}\label{thm:intro_main_classification}
Let $G$ be a finite group acting on a finite tensor category $\cC$. Fiber functors on $\cC^G$ are classified by conjugacy classes of pairs $(H,\cM)$, where $H\subseteq G$ is a subgroup and $\cM$ is an $H$-equivariant $\cC$-module category whose underlying $\cC$-module category is indecomposable, exact, and semisimple, satisfying:
\begin{enumerate}
\item $H$ acts transitively on $\Irr(\cM)$;
\item for some, and hence every, simple object $X\in\cM$, the stabilizer cocycle
\[
[\alpha_X]\in H^2(\St_H(X),\ku^\times)
\]
is non-degenerate.
\end{enumerate}
\end{theorem}

This is Theorem~\ref{thm:main_classification}; the relevant notion of conjugacy is Definition~\ref{def:conjugate_equivariant_module_categories}. In this form, the theorem turns the fiber-functor problem for $\cC^G$ into a rank-one condition for equivariant module categories; via Tannaka--Krein reconstruction it also covers Hopf realizations outside the semisimple setting. Structurally, the simple objects of $\cM^H$ are obtained from $H$-orbits in $\Irr(\cM)$ together with projective representations of stabilizers. Hence $\cM^H$ has rank one exactly when the $H$-action is transitive and the twisted group algebra $\ku_{\alpha_X}[\St_H(X)]$ of a stabilizer is simple. When the relevant second cohomology groups vanish, for instance whenever every subgroup of $G$ is cyclic, the description reduces to freeness and transitivity of the $H$-action on $\Irr(\cM)$.

The proof combines three known facts. Fiber functors correspond to rank-one exact module categories~\cite[Proposition~7.3.3 and Example~7.4.6]{EGNO15}. Exact module categories over equivariantizations come from equivariant module categories over the original category~\cite{GM11}. Finally, simple objects in equivariant semisimple categories are governed by projective representations of stabilizers. Combining these facts gives Theorem~\ref{thm:intro_main_classification}.

Here and throughout, $C_n$ denotes the cyclic group of order $n$.

Our first application concerns a family of semisimple Hopf algebras due to Nikshych.

\begin{theorem}\label{thm:intro_nikshych}
Let $p$ be an odd prime, and let $H_p$ be Nikshych's semisimple Hopf algebra of dimension $4p^2$ constructed in~\cite{Nik08}. Then $\Rep(H_p)$ has one equivalence class of fiber functors if $p\equiv3\pmod4$, and two equivalence classes of fiber functors if $p\equiv1\pmod4$.
\end{theorem}

This is Theorem~\ref{thm:nikshych_fiber_functors}. In Nikshych's construction, $H_p$ is reconstructed from the $C_2$-equivariantization of the Tambara--Yamagami category $\TY(C_p\times C_p,\chi_p,p^{-1})$, where $\chi_p$ is the standard hyperbolic form and $C_2$ acts by interchanging the two factors. This Tambara--Yamagami category has exactly two fiber functors, corresponding to the ordered decompositions $C_p\times C_p=L_1\times L_2$ and $C_p\times C_p=L_2\times L_1$, where $L_1=C_p\times\{0\}$ and $L_2=\{0\}\times C_p$. The involution used in the construction of $H_p$ interchanges them; hence they give one equivalence class on the equivariantization. The remaining possibility left by Theorem~\ref{thm:intro_main_classification}, namely a fiber functor coming from $H=C_2$, is controlled by the rank-two module categories over this Tambara--Yamagami category described in~\cite{Nik08}: it is impossible for $p\equiv3\pmod4$, and contributes one additional class for $p\equiv1\pmod4$.

Our second application concerns gaugings. A gauging of a braided fusion category $\cB$ is an equivariantization $\cE^G$, where $\cE=\bigoplus_{g\in G}\cE_g$ is a faithful $G$-crossed braided fusion category with neutral component $\cE_e=\cB$~\cite{CGPW16, ENO10}. Combining Theorem~\ref{thm:intro_main_classification} with Clifford theory for module categories over graded fusion categories~\cite{Gal12, MM20}, we reduce the module-category input from $\cE$ to module categories over the neutral component $\cB$, together with extension and equivariant-lift data; this is Theorem~\ref{thm:gauging_clifford_fiber_functors}. For prime cyclic gaugings the description is especially explicit: if $G=C_p$, the relevant second cohomology groups vanish, and Corollary~\ref{cor:prime_cyclic_gauging} says that a fiber functor either comes from a rank-one module category over $\cE$, or from a $C_p$-equivariant $\cE$-module category with $p$ simple objects. In the latter case, the module category is either induced from a $\cB$-module category $\cN$ for which every component $\cE_g\boxtimes_{\cB}\cN$ has rank one, or is already an $\cE$-module category of rank $p$; in both subcases the generator acts on simple objects as a single $p$-cycle.

We apply this to the small radical gaugings classified in~\cite{GNradical}. There the Tannakian radical is the subcategory $\Rep(G)$ being gauged, and the mantle is the braided fusion category $\cB$ appearing as the neutral component. The non-pointed, non-degenerate entries are canonical gaugings of their mantles by their radicals. We reproduce the relevant table as Table~\ref{tab:small_radical_gaugings}, where $\cC(A,q)$ denotes the pointed modular category attached to a \emph{metric group} $(A,q)$, that is, a finite abelian group $A$ with a non-degenerate quadratic form $q$, and add the fiber-functor conclusion. Since a non-degenerate braiding on $\Rep(H)$ corresponds to a factorizable quasitriangular structure on the semisimple Hopf algebra $H$, deciding which entries admit a fiber functor is the same as deciding which are realized by finite-dimensional semisimple factorizable Hopf algebras.

\begin{corollary}\label{cor:intro_factorizable_realizations}
Apart from pointed categories, the entries from the small-dimensional list in~\cite{GNradical} considered in Table~\ref{tab:small_radical_gaugings} that admit realizations as representation categories of finite-dimensional semisimple factorizable Hopf algebras are exactly:
\begin{enumerate}
\item the ordinary dihedral doubles $\Rep(D(D_r))$;
\item the hyperbolic prime-cyclic twisted-double rows $\Rep(D^\omega(C_r\rtimes C_p))$, with $p\mid r-1$, for every cohomology class $[\omega]\in H^3(C_r\rtimes C_p,\ku^\times)$.
\end{enumerate}
All remaining non-pointed entries in that table have no such realization.
\end{corollary}

This is Corollary~\ref{cor:factorizable_hopf_realizations}. Its proof draws on several distinct obstructions. The rows containing Ising categories, metaplectic categories, and centers of elliptic Tambara--Yamagami categories are excluded by integrality. For the dihedral twisted doubles, the group-theoretical description of fiber functors, Goursat's lemma, and the explicit cocycle representatives of~\cite{CGR00} show that only the trivial cohomology class survives. The odd elliptic prime-cyclic rows are excluded using the prime-cyclic gauging description, the non-group-theoreticality theorem of~\cite{GPR24}, and the Hopf-algebra obstruction of~\cite{JL09}. The dimension $36$ row requires instead a separate invariant-Lagrangian computation for the $\mathbb{F}_4$ extension, together with the same Hopf-algebra obstruction.

\medskip

The paper is organized as follows. Section~\ref{sec:preliminaries} fixes conventions on finite tensor categories, equivariantizations, and equivariant module categories, and records the conjugacy theorem for equivariantized module categories, Theorem~\ref{thm:conjugacy_equivariantization}. Section~\ref{sec:fiber-functors} proves the main classification, Theorem~\ref{thm:main_classification}. Section~\ref{sec:nikshych} proves Theorem~\ref{thm:nikshych_fiber_functors}. Section~\ref{sec:gauging} establishes the Clifford-theoretic description of fiber functors on gaugings and applies it to the table of~\cite{GNradical}.

\section{Preliminaries}\label{sec:preliminaries}

We use~\cite{EGNO15, EO04} as our main references for finite tensor categories and module categories. All categories and functors are $\ku$-linear, and functors are additive, unless stated otherwise.

\subsection{Finite tensor categories and their module categories}

A \emph{finite tensor category} is a finite abelian category equipped with a rigid monoidal structure, bilinear tensor product, and simple unit~\cite[Section~2.2]{EO04}; a \emph{fusion category} is a semisimple finite tensor category.

A left $\cC$-module category is a finite abelian category $\cM$ with an exact bilinear action $\overline\otimes:\cC\times\cM\to\cM$ and the usual associativity and unit constraints. It is \emph{exact} if $P\overline\otimes M$ is projective whenever $P$ is projective in $\cC$, and \emph{indecomposable} if it is not a direct sum of two nonzero module categories. A $\cC$-module functor $F:\cM\to\cN$ is an exact functor equipped with coherent natural isomorphisms
\[
F(X\overline\otimes M)\xrightarrow{\sim}X\overline\otimes F(M).
\]
We write $\Fun_{\cC}(\cM,\cN)$ for the category of module functors and module natural transformations.

\begin{remark}
If $\cM$ is indecomposable and exact, then
\[
\cC^*_{\cM}:=\Fun_{\cC}(\cM,\cM)
\]
is a finite tensor category~\cite[Theorem~3.27]{EO04}. We use reverse composition, $F\otimes G:=G\circ F$, so that evaluation gives the right action $M\triangleleft F=F(M)$ and the double-dual equivalence $(\cC^*_{\cM})^*_{\cM}\simeq\cC$. In this situation $\cC$ and $\cC^*_{\cM}$ are Morita equivalent.
\end{remark}

\subsection{Equivariant tensor categories and their module categories}

\subsubsection{Group actions on tensor categories}

Let $\underline G$ be the strict monoidal category associated with a group $G$. An action of $G$ on a tensor category $\cC$ is a monoidal functor
\[
\underline G\longrightarrow \underline{\AutT(\cC)}.
\]
By the coherence theorem for tensor categories with $G$-action~\cite[Theorem~2.4]{Gal17}, we may and do assume that the $G$-action is strict. We write $g_*$ for the tensor autoequivalence attached to $g$, so that $g_*h_*=(gh)_*$ and $e_*=\id_{\cC}$. Equivariant module categories will not be strictified; their coherence maps are retained below.

\subsubsection{The semidirect product}

For a tensor category $\cC$ with a strict $G$-action, the \emph{semidirect product} $\cC\rtimes G$ is the monoidal category
\[
\cC\rtimes G=\bigoplus_{g\in G}\cC_g,
\]
where each homogeneous component $\cC_g$ is a copy of $\cC$~\cite{T01,Nik08}. We write $[X,g]$ for the object $X\in\cC$ placed in degree $g$. The tensor product and unit are
\begin{equation}\label{eq:semidirect_product_tensor}
[X,g]\otimes[Y,h]=[X\otimes g_*(Y),gh],
\qquad
\mathbf 1_{\cC\rtimes G}=[\mathbf 1,e],
\end{equation}
with associativity and unit constraints induced from those of $\cC$. When $G$ is finite, $\cC\rtimes G$ is a finite tensor category. For an arbitrary subset $S\subseteq G$, we use the notation
\[
(\cC\rtimes G)_S:=\bigoplus_{s\in S}\cC_s.
\]
If $S$ is a subgroup, this is the tensor subcategory $\cC\rtimes S$. If $S=gK$ is a left coset, then $(\cC\rtimes G)_{gK}$ is naturally a
$((\cC\rtimes G)_{gKg^{-1}},\cC\rtimes K)$-bimodule category under tensor product. This notation will also be used for arbitrary faithfully graded tensor categories.

\subsubsection{$H$-equivariant module categories}

We follow~\cite{Gal11, ENO11} for equivariant module categories and their use in constructing module categories over equivariantizations.

Let $H \subseteq G$ be a subgroup acting on a finite tensor category $\cC$. Given a $\cC$-module category $\cM$, we \emph{twist} the module structure by any element $g\in H$.

For each $g\in H$, let $\cM^g$ be the category with the same underlying abelian category as $\cM$ and with $\cC$-action
\[
X\overline{\otimes}^{\,g}M:=g_*(X)\overline{\otimes}M.
\]
The associativity and unit constraints are those induced from the module structure of $\cM$ and the tensor structure of $g_*$. Since the action on $\cC$ is strict, $(\cM^g)^h=\cM^{gh}$ as $\cC$-module categories.

We now make explicit the monoidal category that records twisted module autoequivalences. Let $\Aut_{\cC}^{H}(\cM)$ be the groupoid whose objects are pairs $(g,F)$, where $g\in H$ and
\[
(F,c^F):\cM\longrightarrow\cM^g
\]
is a $\cC$-module equivalence. There are no morphisms between objects of different degrees, and morphisms in degree $g$ are $\cC$-module natural isomorphisms. If $(g,F)$ and $(h,E)$ are objects, define
\[
(g,F)\otimes(h,E):=(gh,F\circ E),
\]
where the $\cC$-module constraint on the composite is
\begin{equation}\label{eq:twisted_composite_constraint}
\begin{aligned}
c^{F\circ E}_{X,M}:F(E(X\overline\otimes M))
&\xrightarrow{\ F(c^E_{X,M})\ }
F(h_*(X)\overline\otimes E(M))\\
&\xrightarrow{\ c^F_{h_*(X),E(M)}\ }
(gh)_*(X)\overline\otimes F(E(M)).
\end{aligned}
\end{equation}
With the usual horizontal product of module natural isomorphisms, composition makes this a strict monoidal groupoid with unit $(e,\id_{\cM})$. The degree map
\[
\pi_1:\Aut_{\cC}^{H}(\cM)\longrightarrow\underline H,
\qquad
(g,F)\longmapsto g,
\]
is a strict monoidal functor.

An \emph{$H$-equivariant $\cC$-module category} is a $\cC$-module category $\cM$ together with a strong monoidal section
\[
\rho:\underline H\longrightarrow\Aut_{\cC}^{H}(\cM),
\qquad
\pi_1\circ\rho=\id_{\underline H}.
\]
We work with normalized sections. Explicitly, the data consist of $\cC$-module equivalences
\[
(U_h,c^h):\cM\longrightarrow\cM^h,
\qquad h\in H,
\]
and module natural isomorphisms
\[
\mu_{g,h}:U_gU_h\xrightarrow{\sim}U_{gh}
\]
such that
\begin{equation}\label{eq:equivariant_module_pentagon}
(\mu_{gh,k})_M\circ(\mu_{g,h})_{U_k(M)}
=(\mu_{g,hk})_M\circ U_g((\mu_{h,k})_M)
\end{equation}
for all $g,h,k\in H$ and $M\in\cM$, and such that each $\mu_{g,h}$ is a module natural transformation, equivalently
\begin{equation}\label{eq:equivariant_module_compatibility}
\begin{aligned}
&(\id_{(gh)_*(X)}\overline\otimes(\mu_{g,h})_M)
\circ c^g_{h_*(X),U_h(M)}\circ U_g(c^h_{X,M})\\
&\hspace{4cm}=
 c^{gh}_{X,M}\circ(\mu_{g,h})_{X\overline\otimes M}
\end{aligned}
\end{equation}
for all $X\in\cC$ and $M\in\cM$. Normalization means
\begin{equation}\label{eq:equivariant_module_normalization}
U_e=\id_{\cM},\qquad c^e=\id,
\qquad \mu_{e,h}=\mu_{h,e}=\id_{U_h}.
\end{equation}
Every strong monoidal section is monoidally isomorphic to a normalized one, so this convention does not change the equivalence classes under consideration.

Under these conventions, normalized $H$-equivariant $\cC$-module structures on $\cM$ are equivalent to $(\cC\rtimes H)$-module structures extending the original $\cC$-action. The correspondence is given explicitly by
\begin{equation}\label{eq:semidirect_product_module_action}
[X,h]\overline\otimes M:=X\overline\otimes U_h(M).
\end{equation}
The associativity constraint for this action is built from $c^h$ and $\mu_{g,h}$; conversely, the action of $[\mathbf1,h]$ recovers $U_h$, and the module associativity recovers $c^h$ and $\mu_{g,h}$.

An $H$-equivariant $\cC$-module category will be called indecomposable when its underlying $\cC$-module category is indecomposable.

\subsection{The equivariant category \texorpdfstring{$\cC^G$}{CG} and its module categories}

Given a $G$-action on a tensor category $\cC$, one forms the \emph{equivariant category} \texorpdfstring{$\cC^G$}{CG}.

Its objects are pairs $(X, u)$ where $X \in \cC$ and $u = \{u_g: g_*(X) \to X\}_{g \in G}$ is a collection of isomorphisms satisfying $u_{gh} = u_g \circ g_*(u_h)$ for all $g, h \in G$ and $u_e = \id_X$. The tensor product is $(X, u) \otimes (Y, v) = (X \otimes Y, w)$, where $w_g = u_g \otimes v_g$, and the unit object is $(\mathbf{1}, \id_{\mathbf{1}})$.

By~\cite[Section~4.15 and Example~7.12.25]{EGNO15}, if $\cC$ is a finite tensor category with a $G$-action, then $\cC^G$ is also a finite tensor category.

Let $H \subseteq G$ be a subgroup, and let $\cM$ be an $H$-equivariant $\cC$-module category with structure $\rho: \underline{H} \to \Aut_{\cC}^{H}(\cM)$. Forgetting the module functor structures, the functors $U_h$ define an action of $H$ on the underlying abelian category $\cM$. One can therefore form the equivariant category $\cM^H$, and the tensor category $\cC^G$ acts on it.

The category $\cM^H$ is defined analogously to $\cC^G$: objects are pairs $(M, v)$ where $M \in \cM$ and $v = \{v_h: U_h(M) \to M\}_{h \in H}$ satisfy $v_e=\id_M$ and the equivariance condition $v_{h_1h_2} \circ (\mu_{h_1,h_2})_M = v_{h_1} \circ U_{h_1}(v_{h_2})$ for all $h_1,h_2 \in H$. Morphisms are defined similarly to preserve the equivariant structure.

It follows from~\cite[Lemma 3.3]{GM11} that $\cM^H$ is a left $\cC^G$-module category. The action of $(X,u) \in \cC^G$ on $(M,v) \in \cM^H$ is given by $(X,u) \overline{\otimes} (M,v) = (X \overline{\otimes} M, \widetilde{v})$, where $\widetilde{v}_h = (u_h \overline{\otimes} v_h) \circ c^h_{X,M}$ for $h \in H$.

We use the following description of exact module categories over equivariantizations, due to Galindo--Mombelli~\cite[Proposition~3.4]{GM11}.

\begin{proposition}\label{prop:mod_cat_equivariantization}
Let $G$ be a finite group acting on a finite tensor category $\cC$. Then the following hold.
\begin{enumerate}
\item If $\cM$ is an $H$-equivariant $\cC$-module category, then $\cM^H$ is an exact $\cC^G$-module category if and only if $\cM$ is an exact $\cC$-module category.
\item If $\cN$ is an indecomposable exact module category over $\cC^G$, then there exists a subgroup $H \subseteq G$ and an $H$-equivariant $\cC$-module category $\cM$, whose underlying $\cC$-module category is indecomposable and exact, such that $\cN \simeq \cM^H$.
\end{enumerate}
\end{proposition}

\subsubsection{Equivalence classes of equivariant module categories}

We record when two equivariant $\cC$-module categories produce equivalent $\cC^G$-module categories.

Let $L\subseteq G$, and let $\cM$ and $\cN$ be two $L$-equivariant $\cC$-module categories, with structure functors $U_\ell$ and $U'_\ell$, respectively. Recall that an $L$-equivariant equivalence is a $\cC$-module functor $F:\cM\to \cN$ together with module natural isomorphisms $w^\ell: U'_\ell\circ F\to F\circ U_\ell$ satisfying the coherence condition
\[
w^{\ell m}_M\circ(\mu'_{\ell,m})_{F(M)}
=F((\mu_{\ell,m})_M)\circ w^\ell_{U_m(M)}\circ U'_\ell(w^m_M)
\]
for all $\ell,m\in L$ and $M\in \cM$. In the strict case this reduces to the usual formula $w^{\ell m}_M=w^\ell_{U_m(M)}\circ U'_\ell(w^m_M)$. This convention is compatible with the correspondence between $L$-equivariant $\cC$-module categories and $(\cC\rtimes L)$-module categories.

For $g\in G$, the rule
\[
\Ad_g([V,h])=[g_*(V),ghg^{-1}]
\]
defines a tensor autoequivalence of $\cC\rtimes G$ and restricts to
$\cC\rtimes K\simeq\cC\rtimes gKg^{-1}$. If $\cN$ is a $K$-equivariant
$\cC$-module category and $H=gKg^{-1}$, let ${}^g\cN$ be the
$H$-equivariant module category obtained by pulling the corresponding
$(\cC\rtimes K)$-module structure back along $\Ad_{g^{-1}}$. On the
underlying category,
\[
V\bullet^g N=(g^{-1})_*(V)\triangleright N,
\qquad
U'_h=U_{g^{-1}hg},
\]
with module constraints and coherence maps transported from those of $\cN$.

\begin{definition}\label{def:conjugate_equivariant_module_categories}
Two pairs $(H,\cM)$ and $(K,\cN)$, where $\cM$ is an $H$-equivariant $\cC$-module category and $\cN$ is a $K$-equivariant $\cC$-module category, are called \emph{conjugate} if there exists $g \in G$ such that
\begin{enumerate}
\item $H = gKg^{-1}$, and
\item $\cM$ is equivalent to ${}^g\cN$ as $H$-equivariant $\cC$-module categories.
\end{enumerate}
\end{definition}

The equivalence relation among these equivariantized module categories is the following.

\begin{theorem}\label{thm:conjugacy_equivariantization}
Let $G$ be a finite group acting on a finite tensor category $\cC$. Let $H,K \subseteq G$ be subgroups. Let $\cM$ be an $H$-equivariant $\cC$-module category and let $\cN$ be a $K$-equivariant $\cC$-module category. Assume that the underlying $\cC$-module categories of $\cM$ and $\cN$ are indecomposable and exact. Then the equivariantizations $\cM^H$ and $\cN^K$ are equivalent as $\cC^G$-module categories if and only if $(H,\cM)$ and $(K,\cN)$ are conjugate.
\end{theorem}

\begin{proof}
Let $\cD=\cC\rtimes G$, with its canonical faithful $G$-grading, and use $\cD_S$ for the sum of the homogeneous components indexed by a subset $S\subseteq G$. The $H$- and $K$-equivariant structures make $\cM$ and $\cN$ into $\cD_H$- and $\cD_K$-module categories. By~\cite[Proposition 3.4(1)]{GM11}, they are indecomposable exact as $\cD_H$- and $\cD_K$-module categories, respectively. Their restrictions to the neutral component $\cD_e=\cC$ are the original indecomposable $\cC$-module categories. Hence $(H,\cM)$ and $(K,\cN)$ are type-1 pairs in the terminology of~\cite[Subsection~5.1]{MM20}.

By~\cite[Proposition 3.4(3)]{GM11} and the argument proving~\cite[Proposition 3.4(5)]{GM11} via Tambara's correspondence~\cite[Theorem 4.1]{T01}, the categories $\cM^H$ and $\cN^K$ are equivalent as $\cC^G$-module categories if and only if the induced $\cD$-module categories
\begin{align*}
\Ind_{\cD_H}^{\cD}(\cM) &:=
\cD\boxtimes_{\cD_H}\cM, \\
\Ind_{\cD_K}^{\cD}(\cN) &:=
\cD\boxtimes_{\cD_K}\cN
\end{align*}
are equivalent.

By~\cite[Definition~5.10 and Proposition~5.12]{MM20}, applied to the $G$-graded tensor category $\cD$, these induced module categories are equivalent if and only if there exists $g\in G$ such that:
\begin{enumerate}
\item $H = gKg^{-1}$, and
\item $\cD_{gK}\boxtimes_{\cD_K}\cN \simeq \cM$ as $\cD_H$-module categories.
\end{enumerate}
The functor
\[
\Phi_g:\cD_{gK}\longrightarrow\cD_K,
\qquad
[V,gk]\longmapsto[(g^{-1})_*(V),k],
\]
induced by left tensoring with $[\mathbf{1},g^{-1}]$, is an equivalence of $(\cD_H,\cD_K)$-bimodule categories, where the left $\cD_H$-action on $\cD_K$ is transported along $\Ad_{g^{-1}}:\cD_H\to\cD_K$. Consequently, the relative tensor product is $\cN$ with its $\cD_H$-action pulled back along $\Ad_{g^{-1}}$, and hence
\[
\cD_{gK}\boxtimes_{\cD_K}\cN\simeq{}^g\cN
\]
as $\cD_H$-module categories. By construction, the right-hand side is precisely the $H$-equivariant $\cC$-module category introduced above.
Thus, by~\cite[Def.~5.10, Prop.~5.12]{MM20}, this equivalence relation is precisely the conjugacy condition above.
\end{proof}

\section{Fiber functors for \texorpdfstring{$\cC^G$}{CG}}\label{sec:fiber-functors}

\subsection{The correspondence between fiber functors and rank-one module categories}

A \emph{fiber functor} on a finite tensor category $\cC$ is an exact faithful tensor functor $F:\cC\to\VecCat_{\ku}$. We use the monoidal constraints
\[
F_2:F(X)\otimes F(Y)\xrightarrow{\sim}F(X\otimes Y),
\qquad
F_0:\ku\xrightarrow{\sim}F(\mathbf1),
\]
and identify fiber functors that are monoidally naturally isomorphic. A $\cC$-module category has \emph{rank one} if its underlying abelian category is equivalent to $\VecCat_{\ku}$.

\begin{proposition}\label{prop:fiber_rank_one}
There is a bijection between equivalence classes of fiber functors on $\cC$ and equivalence classes of exact $\cC$-module categories of rank one.
\end{proposition}

\begin{proof}
Module structures on $\VecCat_{\ku}$ correspond to tensor functors $\cC\to\VecCat_{\ku}$~\cite[Proposition~7.3.3 and Example~7.4.6]{EGNO15}. Such functors are exact and faithful, and the module category is exact because $\VecCat_{\ku}$ is semisimple. Transport along an equivalence with $\VecCat_{\ku}$ gives the converse, while module equivalences correspond to monoidal natural isomorphisms.
\end{proof}

\subsection{Rank-one equivariantizations}

We next characterize when $\cM^H$ has rank one. Since $H$ is finite and
$\operatorname{char}(\ku)=0$, Tambara's theorem gives
\[
\cM^H\text{ semisimple}\quad\Longleftrightarrow\quad \cM\text{ semisimple}
\]
\cite[Proposition~5.2]{T01}. We therefore assume that $\cM$ is semisimple.

\subsubsection{Action on simple objects and obstruction cocycles}

Let $H$ be a finite group acting on a semisimple $\ku$-linear category $\cM$ via autoequivalences $\{U_h\}_{h \in H}$ and normalized coherence isomorphisms $\mu_{h_1,h_2}:U_{h_1}U_{h_2}\xrightarrow{\sim}U_{h_1h_2}$. The group $H$ acts naturally on the set $\Irr(\cM)$ of isomorphism classes of simple objects by $h \cdot [X] := [U_h(X)]$.

For a simple object $X \in \cM$, we denote by $\St_H(X) = \{h \in H \mid U_h(X) \cong X\}$ its \emph{stabilizer subgroup}. For each $h \in \St_H(X)$, we choose an isomorphism $f_h: U_h(X) \xrightarrow{\sim} X$ with $f_e = \id_X$.

\begin{definition}\label{def:obstruction_cocycle}
The \emph{obstruction 2-cocycle} $\alpha_X: \St_H(X) \times \St_H(X) \to \ku^\times$ is defined by
\begin{equation}
    \alpha_X(h_1, h_2) \cdot \id_X =
    f_{h_1} \circ U_{h_1}(f_{h_2}) \circ
    (\mu_{h_1,h_2})_X^{-1}\circ f_{h_1 h_2}^{-1}.
\end{equation}
The right-hand side is an endomorphism of the simple object $X$, hence a scalar multiple of $\id_X$. The cocycle is normalized because $f_e=\id_X$ and the equivariant structure is normalized. When the action is strict, this reduces to the formula with $\mu_{h_1,h_2}=\id$.
\end{definition}

\begin{proposition}\label{prop:obstruction_cocycle_properties}
The map $\alpha_X$ is a 2-cocycle in $Z^2(\St_H(X), \ku^\times)$. Its cohomology class $[\alpha_X] \in H^2(\St_H(X), \ku^\times)$ is independent of the choice of isomorphisms $\{f_h\}$. If $Y$ lies in the $H$-orbit of $X$, then the classes $[\alpha_X]$ and $[\alpha_Y]$ correspond under the conjugation identification of their stabilizers.
\end{proposition}

\begin{proof}
The defining equation can be written as
\begin{equation}\label{eq:obstruction_factor_set}
f_g\circ U_g(f_h)\circ(\mu_{g,h})_X^{-1}
=\alpha_X(g,h)f_{gh}
\end{equation}
for $g,h\in\St_H(X)$. Compare the two composites from
$U_gU_hU_k(X)$ to $X$ obtained by first applying
\eqref{eq:obstruction_factor_set} to $(g,h)$ or to $(h,k)$. The pentagon identity
\eqref{eq:equivariant_module_pentagon} gives
\[
\alpha_X(g,h)\alpha_X(gh,k)
=\alpha_X(g,hk)\alpha_X(h,k),
\]
and normalization gives $\alpha_X(e,h)=\alpha_X(h,e)=1$. Thus
$\alpha_X\in Z^2(\St_H(X),\ku^\times)$.

If $f'_h=\lambda_hf_h$, where $\lambda_h\in\ku^\times$ and
$\lambda_e=1$, then substitution in Definition~\ref{def:obstruction_cocycle} yields
\[
\alpha'_X(g,h)
=\frac{\lambda_g\lambda_h}{\lambda_{gh}}\,\alpha_X(g,h).
\]
Hence $\alpha'_X$ and $\alpha_X$ differ by a coboundary.

Finally, suppose first that $Y=U_a(X)$. Conjugation
$\iota_a(s)=asa^{-1}$ identifies $\St_H(X)$ with $\St_H(Y)$. For
$s\in\St_H(X)$ define
\begin{equation}\label{eq:transported_stabilizer_isomorphism}
\widetilde f_{\iota_a(s)}
:=U_a(f_s)\circ(\mu_{a,s})_X^{-1}
\circ(\mu_{\iota_a(s),a})_X:
U_{\iota_a(s)}(Y)\longrightarrow Y.
\end{equation}
A second application of the pentagon identity shows that, for these transported choices,
\[
\alpha_Y(asa^{-1},ata^{-1})=\alpha_X(s,t)
\qquad(s,t\in\St_H(X)).
\]
If $Y$ is only isomorphic to $U_a(X)$, conjugating the maps in
\eqref{eq:transported_stabilizer_isomorphism} by a chosen isomorphism
$U_a(X)\xrightarrow{\sim}Y$ changes neither conclusion; a different choice changes the cocycle only by a coboundary. Therefore the two cohomology classes correspond under conjugation.
\end{proof}

Recall that a cohomology class $[\alpha] \in H^2(G, \ku^\times)$ is called \emph{nondegenerate} if the corresponding twisted group algebra $\ku_\alpha[G]$ is simple.

We also use the description of simple objects in equivariant semisimple categories, in the form recalled in~\cite[Subsection~7.5.1]{GP17}.

\begin{proposition}\label{prop:structure_MH}
Let $H$ be a finite group acting on a finite semisimple $\ku$-linear category $\cM$. If $\{X_1, \ldots, X_m\}$ is a set of representatives of the $H$-orbits of simple objects in $\cM$, then there is a bijective correspondence between simple objects in $\cM^H$ and pairs $(i, V)$ where $1 \leq i \leq m$ and $V$ is a simple right module over the twisted group algebra $\ku_{\alpha_{X_i}}[\St_H(X_i)]$.
\end{proposition}

\begin{corollary}\label{cor:rank_one_MH}
Let $H$ be a finite group acting on a finite semisimple $\ku$-linear category $\cM$. The category $\cM^H$ has rank one if and only if:
\begin{enumerate}
    \item The action of $H$ on $\Irr(\cM)$ is transitive.
    \item For a simple object $X \in \cM$, the obstruction cocycle $[\alpha_X] \in H^2(\St_H(X), \ku^\times)$ is nondegenerate.
\end{enumerate}
\end{corollary}

\begin{proof}
By Proposition~\ref{prop:structure_MH}, the number of simple objects of $\cM^H$ is
\[
\sum_{\mathcal{O}\in H\backslash \Irr(\cM)}
\#\,\Irr\big(\ku_{\alpha_X}[\St_H(X)]\big),
\]
where $X$ is any representative of the orbit $\mathcal{O}$. This number is one if and only if there is a single $H$-orbit and the twisted group algebra attached to the stabilizer has a single simple module. Since the algebra is finite-dimensional semisimple over the algebraically closed field $\ku$, this is equivalent to the twisted group algebra being simple, which is the nondegeneracy condition for $[\alpha_X]$.
\end{proof}

\subsection{Classification of fiber functors}

\begin{theorem}\label{thm:main_classification}
Let $G$ act on a finite tensor category $\cC$. Fiber functors on $\cC^G$ are in bijection with conjugacy classes of pairs $(H,\cM)$ such that:
\begin{enumerate}
\item $H\leq G$;
\item $\cM$ is an $H$-equivariant $\cC$-module category whose underlying $\cC$-module category is indecomposable, exact, and semisimple;
\item $H$ acts transitively on $\Irr(\cM)$;
\item for some, equivalently every, simple $X\in\cM$, the class $[\alpha_X]\in H^2(\St_H(X),\ku^\times)$ is nondegenerate.
\end{enumerate}
\end{theorem}

\begin{proof}
By Proposition~\ref{prop:fiber_rank_one}, a fiber functor on $\cC^G$ determines an exact $\cC^G$-module category $\cN$ of rank one. In particular, $\cN$ is indecomposable and semisimple. Proposition~\ref{prop:mod_cat_equivariantization} gives a subgroup $H\subseteq G$ and an $H$-equivariant $\cC$-module category $\cM$, with indecomposable exact underlying $\cC$-module category, such that $\cN\simeq\cM^H$. Since $H$ is finite and $|H|$ is invertible in $\ku$, Tambara's semisimplicity theorem~\cite[Proposition~5.2]{T01} implies that $\cM$ is semisimple. Corollary~\ref{cor:rank_one_MH} then gives transitivity of the action on $\Irr(\cM)$ and nondegeneracy of the stabilizer cocycle.

Conversely, let $(H,\cM)$ satisfy the conditions in the statement. Part~(1) of Proposition~\ref{prop:mod_cat_equivariantization} shows that $\cM^H$ is an exact $\cC^G$-module category. Since $\cM$ is semisimple, items~\textup{(3)--(4)} and Corollary~\ref{cor:rank_one_MH} imply that $\cM^H$ has rank one. Proposition~\ref{prop:fiber_rank_one} therefore produces a fiber functor on $\cC^G$.

Finally, Theorem~\ref{thm:conjugacy_equivariantization} says that two such pairs yield equivalent $\cC^G$-module categories, and hence equivalent fiber functors, exactly when they are conjugate. This proves both directions and the asserted bijection.
\end{proof}

\begin{corollary}\label{cor:cohomologically_trivial}
Assume that $H^2(L,\ku^\times)=0$ for every subgroup $L\leq G$. Then fiber functors on $\cC^G$ are classified by conjugacy classes of pairs $(H,\cM)$ satisfying items \textup{(1)--(2)} of Theorem~\ref{thm:main_classification} and such that $H$ acts freely and transitively on $\Irr(\cM)$. This applies in particular when every subgroup of $G$ is cyclic, hence when $G$ is cyclic.
\end{corollary}

\begin{proof}
Under the hypothesis, every stabilizer cocycle is cohomologically trivial. Its twisted group algebra is therefore an ordinary group algebra, which is simple over $\ku$ exactly when the stabilizer is trivial. Theorem~\ref{thm:main_classification} gives the result.
\end{proof}

\section{Fiber functors for the Hopf algebras \texorpdfstring{$H_p$}{Hp}}\label{sec:nikshych}

We apply Theorem~\ref{thm:main_classification} to Nikshych's $C_2$-equivariantizations of Tambara--Yamagami categories~\cite{Nik08}.

\subsection{Tambara--Yamagami categories and fiber functors}

For a finite abelian group $A$, a nondegenerate symmetric bicharacter $\chi$, and $\tau^2=|A|^{-1}$, the category $\TY(A,\chi,\tau)$ has simple objects $A\cup\{m\}$ and fusion rules
\[
a\otimes b=a+b,
\qquad
a\otimes m=m\otimes a=m,
\qquad
m\otimes m=\bigoplus_{a\in A}a;
\]
its associator is determined by $(\chi,\tau)$~\cite{TY98}. In particular, $\FPdim(m)=\sqrt{|A|}$ and $\FPdim(\TY(A,\chi,\tau))=2|A|$.

The form $\chi$ is \emph{hyperbolic} if $A=L\times L'$ with $\chi$ trivial on $L\times L$ and on $L'\times L'$; otherwise it is \emph{elliptic}. For $|A|$ odd, Tambara's classification states that $\TY(A,\chi,\tau)$ admits a fiber functor exactly when $\tau^{-1}$ is a positive integer and $\chi$ is hyperbolic. Its fiber functors are then parametrized by the ordered decompositions $(L,L')$ of this type~\cite[Propositions~4.1 and~4.2]{T00}.

\subsection{The Hopf algebras \texorpdfstring{$H_p$}{Hp}}

Let $p$ be an odd prime, $A_p=C_p\times C_p$, and choose a primitive $p$th root of unity $\xi\in\ku$. We identify $\mathbb F_p$-valued bilinear forms with bicharacters by exponentiating with $\xi$, and use the convention
\[
\Alt(\mu)(x,y)=\mu(y,x)\mu(x,y)^{-1},
\qquad
J=\begin{pmatrix}0&1\\-1&0\end{pmatrix}.
\]
Set
\[
\chi_p((x_1,x_2),(y_1,y_2))=\xi^{x_1y_2+y_1x_2},
\qquad
\cC_p=\TY(A_p,\chi_p,p^{-1}).
\]
The only Lagrangian subgroups of $(A_p,\chi_p)$ are
$L_1=C_p\times\{0\}$ and $L_2=\{0\}\times C_p$. Hence $\cC_p$ has two fiber functors, represented by the ordered decompositions $(L_1,L_2)$ and $(L_2,L_1)$.

The involution $t(x,y)=(y,x)$ preserves $\chi_p$ and induces a strict $C_2$-action on $\cC_p$ with $T_t(a)=t(a)$ and $T_t(m)=m$~\cite[Proposition~2.10]{Nik08}. Nikshych's Hopf algebra $H_p$ is characterized by
\[
\Rep(H_p)\simeq\cC_p^{C_2};
\]
see~\cite[Theorem~1]{CM17} for an explicit construction.

\subsection{\texorpdfstring{Classification of fiber functors for $\Rep(H_p)$}{Classification of fiber functors for Rep(Hp)}}

The only subgroups in Theorem~\ref{thm:main_classification} are $J=\{e\}$ and $J=C_2$. The first gives the $C_2$-orbits of fiber functors on $\cC_p$. In the second, Corollary~\ref{cor:cohomologically_trivial} requires an indecomposable rank-two $\cC_p$-module category with a free transitive $C_2$-action on its simples.

\begin{lemma}\label{lem:no_rank_two_invariant_nikshych}
The \(T_t\)-invariant indecomposable rank-two \(\cC_p\)-module categories are as follows. If \(p\equiv3\pmod4\), there are none. If \(p\equiv1\pmod4\), there is exactly one equivalence class.
\end{lemma}

\begin{proof}
With the convention fixed above, the map
\([\mu]\mapsto\Alt(\mu)\) identifies
\(H^2(A_p,\ku^\times)\) with the alternating bicharacters on \(A_p\).
Tambara uses the reciprocal convention in~\cite[Proposition~2.6]{T00},
which gives the same parametrization.

By~\cite[Lemma~4.3]{Nik08}, every indecomposable rank-two
\(\cC_p\)-module category is of the form
\[
\cN_\mu=\Mod_{\cC_p}(R(A_p,\mu)),
\]
where \([\mu]\ne0\) in \(H^2(A_p,\ku^\times)\). Since \(A_p\) has
dimension two, every non-zero alternating form on \(A_p\) is
nondegenerate. Thus there is a unique \(\lambda\in\mathbb F_p^\times\)
such that
\[
\Alt(\mu)=\lambda J.
\]

We first determine the equivalence class of \(\cN_\mu\). By
\cite[Remark~4.5]{Nik08}, restriction to the pointed subcategory gives
\[
\cN_\mu|_{\Vect{A_p}}
\simeq
\ModCat(A_p,\mu)\oplus\ModCat(A_p,\mu'),
\]
where, if \(S\) is the matrix of the symmetric bicharacter \(\chi_p\), the
automorphism \(\iota_\mu\) defined in that remark satisfies
\[
\lambda J\,\iota_\mu=S,
\qquad
\mu'(x,y)=\mu^{-1}(\iota_\mu(x),\iota_\mu(y)).
\]
Since \(\iota_\mu=-\lambda^{-1}JS\) and
\(S^TJS=\det(S)J\), we obtain
\[
\Alt(\mu')
=\iota_\mu^T(-\lambda J)\iota_\mu
=-\frac{\det(S)}{\lambda}J.
\]
This is also the specialization of the translate formula in
\cite[Lemma~9.1]{MeMu12}. For the hyperbolic form \(\chi_p\),
\[
S=\begin{pmatrix}0&1\\1&0\end{pmatrix},
\qquad \det(S)=-1,
\]
and hence
\[
\Alt(\mu')=\lambda^{-1}J.
\]

Consequently, the restriction of \(\cN_\mu\) is encoded by the unordered
pair
\[
\{\lambda,\lambda^{-1}\}.
\]
This unordered pair is a complete invariant. Indeed, an equivalence of
\(\cC_p\)-module categories restricts to an equivalence of
\(\Vect{A_p}\)-module categories and therefore identifies the two
indecomposable summands, possibly after interchanging them. Conversely,
let \(x,y\) be the two simple objects of \(\cN_\mu\), with
\(\underline{\operatorname{End}}(x)\simeq R(A_p,\mu)\). The restriction
formula identifies the neutral part of
\(\underline{\operatorname{End}}(y)\) with \(R(A_p,\mu')\). The fusion
rules in~\cite[Lemma~4.3]{Nik08} give
\(m\overline\otimes y\simeq p x\), so the non-trivial graded component of
\(\underline{\operatorname{End}}(y)\) vanishes. Therefore
\[
\underline{\operatorname{End}}(y)\simeq R(A_p,\mu').
\]
Internal-Hom reconstruction~\cite[Section~2.1]{Nik08} now gives
\[
\cN_\mu
\simeq
\Mod_{\cC_p}(R(A_p,\mu'))
=\cN_{\mu'}.
\]
Since the cohomology class of a cocycle on \(A_p\) is determined by its
alternating form, the unordered pair is both necessary and sufficient.
Thus the equivalence classes are parametrized by
\[
\{\lambda,\lambda^{-1}\},
\qquad \lambda\in\mathbb F_p^\times.
\]

The autoequivalence \(T_t\), induced by \(t(x,y)=(y,x)\), sends
\(\cN_\mu\) to \(\cN_{t\mu}\). Since
\[
t^TJt=-J,
\]
it sends \(\lambda\) to \(-\lambda\). Hence \(\cN_\mu\) is
\(T_t\)-invariant if and only if
\[
\{\lambda,\lambda^{-1}\}
=
\{-\lambda,-\lambda^{-1}\}.
\]
As \(p\) is odd and \(\lambda\ne0\), this is equivalent to
\[
\lambda=-\lambda^{-1},
\qquad\text{or equivalently}\qquad
\lambda^2=-1.
\]

If \(p\equiv3\pmod4\), there is no solution. If
\(p\equiv1\pmod4\), the two solutions \(i\) and \(-i\) determine the
same unordered pair because \(i^{-1}=-i\). Hence there is exactly one
\(T_t\)-invariant equivalence class.
\end{proof}

\begin{lemma}\label{lem:diagonal_ty_subcategory}
Let $s$ denote the non-trivial element of the acting group $C_2$; it acts on
$\cC_p$ by $T_t$. In the semidirect product $\cD=\cC_p\rtimes C_2$, let
\[
\widetilde m=[m,s],
\qquad
\cE_\Delta=\langle [a,e]\;(a\in A_p),\widetilde m\rangle.
\]
Then $\cE_\Delta$ is a Tambara--Yamagami fusion subcategory
\[
\cE_\Delta\simeq \TY(A_p,\widetilde\chi_p,p^{-1}),
\qquad
\widetilde\chi_p(a,b)=\chi_p(a,t(b)).
\]
In coordinates,
\[
\widetilde\chi_p((x_1,x_2),(y_1,y_2))
=
\xi^{x_1y_1+x_2y_2}.
\]
Thus $\widetilde\chi_p$ is hyperbolic if and only if $p\equiv1\pmod4$.
\end{lemma}

\begin{proof}
Using the strict semidirect-product convention
\(
[X,g]\otimes[Y,h]=[X\otimes g_*(Y),gh]
\), we obtain
\[
\begin{aligned}
\bigl[a,e\bigr]\otimes[b,e]&=[a+b,e],\\
\bigl[a,e\bigr]\otimes\widetilde m&\simeq\widetilde m
\simeq\widetilde m\otimes[a,e],\\
\widetilde m\otimes\widetilde m
&\simeq\bigoplus_{a\in A_p}[a,e].
\end{aligned}
\]
Moreover, \([a,e]^*=[-a,e]\) and \(\widetilde m^*\simeq\widetilde m\). Thus
\(\cE_\Delta\) has Tambara--Yamagami fusion rules.

The action used to form the semidirect product is strict with identity tensor
structure~\cite[Proposition~2.10]{Nik08}, so the relevant associators are
inherited from \(\cC_p\). For \(a,b\in A_p\), they are
\[
\begin{aligned}
\alpha_{\,[a,e],\widetilde m,\,[b,e]}
 &=\chi_p(a,t(b))\,\id_{\widetilde m},\\
\alpha_{\widetilde m,[a,e],\widetilde m}
 &=\bigoplus_{b\in A_p}\chi_p(a,t(b))\,\id_{[b,e]},\\
(\alpha_{\widetilde m,\widetilde m,\widetilde m})_{a,b}
 &=p^{-1}\chi_p(a,t(b))^{-1}.
\end{aligned}
\]
Here we used that \(t\) preserves \(\chi_p\). These are precisely the
associativity data of
\[
\TY(A_p,\widetilde\chi_p,p^{-1}),
\qquad
\widetilde\chi_p(a,b)=\chi_p(a,t(b)).
\]

In coordinates,
\[
\widetilde\chi_p((x_1,x_2),(y_1,y_2))
 =\xi^{x_1y_1+x_2y_2},
\]
so \(\widetilde\chi_p\) is represented by the identity matrix. It is
hyperbolic precisely when \(x_1^2+x_2^2=0\) has a nonzero solution, or
equivalently when \(-1\) is a square in \(\mathbb F_p\). This is equivalent
to \(p\equiv1\pmod4\); if \(i^2=-1\), the two Lagrangian lines are
\(\langle(1,i)\rangle\) and \(\langle(1,-i)\rangle\).
\end{proof}

\begin{proposition}\label{prop:rank_two_equivariant_p_one_mod_four}
Assume that $p\equiv1\pmod4$. The unique $T_t$-invariant indecomposable
rank-two $\cC_p$-module category admits a $C_2$-equivariant structure whose
non-trivial element acts freely on its two simple objects. Moreover, the case $J=C_2$ contributes exactly one equivalence class of fiber functors
on $\cC_p^{C_2}$.
\end{proposition}

\begin{proof}
Let
\[
\cD=\cC_p\rtimes C_2.
\]
We use the faithful grading
\[
K:=\mathbb Z/2\mathbb Z\times C_2,
\]
where the first factor is the Tambara--Yamagami grading of $\cC_p$ and the
second is the semidirect-product grading. Put
\[
\Delta=\{(0,e),(1,s)\}\subset K,
\qquad
\gamma=(0,s),
\]
and let
\[
\cE_\Delta=\cD_{(0,e)}\oplus\cD_{(1,s)}.
\]
By Lemma~\ref{lem:diagonal_ty_subcategory},
\[
\cE_\Delta\simeq
\TY(A_p,\widetilde\chi_p,p^{-1}),
\qquad
\widetilde\chi_p(a,b)=\chi_p(a,t(b)).
\]
Since $p\equiv1\pmod4$, the form $\widetilde\chi_p$ is hyperbolic. If
$i^2=-1$ in $\mathbb F_p$, its two Lagrangian lines are
\[
L_+=\langle(1,i)\rangle,
\qquad
L_-=\langle(1,-i)\rangle.
\]
Tambara's classification~\cite[Propositions~4.1 and~4.2]{T00} therefore gives
exactly two rank-one $\cE_\Delta$-module categories, denoted $\cN_+$ and
$\cN_-$, corresponding to the two ordered decompositions
$A_p=L_+\times L_-$ and $A_p=L_-\times L_+$. Fix either one and denote it by
$\cN$.

Consider the induced $\cD$-module category
\[
\widetilde{\cM}
:=
\Ind_{\cE_\Delta}^{\cD}(\cN)
=
\cD\boxtimes_{\cE_\Delta}\cN.
\]
By~\cite[Propositions~3.5 and~3.7]{Gal12}, the induced category is
graded by the two right cosets $\Delta$ and $\gamma\Delta$. Accordingly,
as a right $\cE_\Delta$-module category,
\[
\cD=\cE_\Delta\oplus\cD_{\gamma\Delta}.
\]
The object
\[
u=[\mathbf1,s]\in\cD_\gamma
\]
is invertible, and left tensoring by $u$ gives an equivalence of right
$\cE_\Delta$-module categories
\[
\cE_\Delta\longrightarrow\cD_{\gamma\Delta},
\qquad
X\longmapsto u\otimes X.
\]
Consequently,
\[
\widetilde{\cM}
\simeq
\cN\oplus\cN
\]
as a semisimple category. In particular, $\widetilde{\cM}$ has exactly two
simple objects, one in each coset summand. The invertible object $u$ exchanges
these two summands, so $\widetilde{\cM}$ is indecomposable as a $\cD$-module
category.

Let
\[
H=\mathbb Z/2\mathbb Z\times\{e\}\subset K;
\qquad
\cD_H=\cC_p.
\]
The homogeneous component of degree $(1,e)$ sends the coset $\Delta$ to
$(1,e)\Delta=\gamma\Delta$ and conversely. Hence the action of the non-zero
object $[m,e]\in\cD_{(1,e)}$ sends either simple object of
$\widetilde{\cM}$ to a non-zero multiple of the other. It follows that
\[
\Res_{\cC_p}^{\cD}(\widetilde{\cM})
\]
is an indecomposable rank-two $\cC_p$-module category.

A module category over $\cC_p\rtimes C_2$ is equivalently a $C_2$-equivariant
$\cC_p$-module category. Thus the restricted module category is
$T_t$-invariant, and Lemma~\ref{lem:no_rank_two_invariant_nikshych} identifies
it with the unique such rank-two $\cC_p$-module category. Moreover, the
non-trivial element of $C_2$ acts through $u=[\mathbf1,s]$, which exchanges
the two simple objects. The action is therefore free and transitive. Since
$H^2(C_2,\ku^\times)=0$, Corollary~\ref{cor:cohomologically_trivial} shows
that its equivariantization has rank one, and hence defines a fiber functor on
$\cC_p^{C_2}$.

It remains to prove uniqueness in the case $J=C_2$. Let $\cM$ be any
$C_2$-equivariant indecomposable rank-two $\cC_p$-module category whose
non-trivial element acts freely on simple objects, and regard $\cM$ as a
$\cD$-module category. Let $x,y$ be its two simple objects. By the fusion-rule
calculation in~\cite[Lemma~4.3]{Nik08}, every $a\in A_p$ fixes both $x$ and
$y$, whereas the non-invertible component of $\cC_p$ interchanges them.
Thus the restriction of $\cM$ to
\[
\cD_{(0,e)}=\Vect{A_p}
\]
has two indecomposable rank-one summands, generated by $x$ and $y$. The
degree $(1,e)$ component interchanges these summands, and so does the degree
$(0,s)$ component because the given $C_2$-action is free. Their product degree
$(1,s)$ therefore stabilizes each summand. Hence the stabilizer in $K$ of either
$\Vect{A_p}$-summand is precisely $\Delta$.

The Clifford theorem for graded fusion categories~\cite[Corollary~4.4]{Gal12}
now gives
\[
\cM\simeq
\cD\boxtimes_{\cE_\Delta}\overline{\cN},
\]
where $\overline{\cN}$ is a $\cE_\Delta$-module category whose restriction to
$\Vect{A_p}$ is one of the two rank-one summands above. In particular,
$\overline{\cN}$ itself has rank one, so it is equivalent to $\cN_+$ or
$\cN_-$.

Finally, the coset degree $\gamma=(0,s)$ normalizes $\Delta$. Conjugation by
$u=[\mathbf1,s]$ acts on $\cE_\Delta$ by $t$ on $A_p$ and fixes
$\widetilde m=[m,s]$. Since
\[
t(L_+)=L_-,
\qquad
t(L_-)=L_+,
\]
the translated module category
\[
\cD_{\gamma\Delta}\boxtimes_{\cE_\Delta}\cN_+
\]
is equivalent to $\cN_-$ as an $\cE_\Delta$-module category. The equivalence
criterion for induced module categories~\cite[Proposition~4.6]{Gal12}
therefore implies
\[
\cD\boxtimes_{\cE_\Delta}\cN_+
\simeq
\cD\boxtimes_{\cE_\Delta}\cN_-.
\]
Thus the case $J=C_2$ contributes exactly one equivalence class of fiber
functors on $\cC_p^{C_2}$.
\end{proof}

\begin{theorem}\label{thm:nikshych_fiber_functors}
Let $p$ be an odd prime, and let $H_p$ be Nikshych's semisimple Hopf algebra
of dimension $4p^2$, so that
\[
\Rep(H_p)\simeq \cC_p^{C_2}
\]
as tensor categories. Then $\Rep(H_p)$ has one equivalence class of fiber
functors if $p\equiv3\pmod4$, and two equivalence classes of fiber functors
if $p\equiv1\pmod4$.
\end{theorem}
\begin{proof}
Since every subgroup of $C_2$ is cyclic, its second cohomology with
coefficients in $\ku^\times$ is trivial. Hence
Corollary~\ref{cor:cohomologically_trivial} applies.

We distinguish the two possible subgroups $J\subseteq C_2$.

First let $J=\{e\}$. In this case the relevant data are the fiber functors of
$\cC_p$, modulo the action of $C_2$. By Tambara's classification, the two
fiber functors of $\cC_p=\TY(A_p,\chi_p,p^{-1})$ correspond to the two ordered
Lagrangian decompositions
\[
A_p=L_1\times L_2
\qquad\text{and}\qquad
A_p=L_2\times L_1.
\]
The involution $t(x,y)=(y,x)$ exchanges $L_1$ and $L_2$, and hence exchanges
these two ordered decompositions. Therefore the case $J=\{e\}$ contributes
one equivalence class of fiber functors on $\cC_p^{C_2}$.

Now let $J=C_2$. Since $H^2(C_2,\ku^\times)=0$, the rank-one condition in
Corollary~\ref{cor:cohomologically_trivial} requires a $C_2$-equivariant
indecomposable $\cC_p$-module category whose non-trivial element acts freely
and transitively on its simple objects. Thus the underlying $\cC_p$-module
category must have rank two.

If $p\equiv3\pmod4$, Lemma~\ref{lem:no_rank_two_invariant_nikshych} shows
that there is no $T_t$-invariant indecomposable rank-two $\cC_p$-module
category. Hence the case $J=C_2$ contributes no fiber functor, and the single
class from the case $J=\{e\}$ is the whole list.

If $p\equiv1\pmod4$, Proposition~\ref{prop:rank_two_equivariant_p_one_mod_four}
shows that the unique $T_t$-invariant indecomposable rank-two
$\cC_p$-module category admits a $C_2$-equivariant structure whose
non-trivial element acts freely on its two simple objects, and that the case
$J=C_2$ contributes exactly one equivalence class of fiber functors.

Finally, the two classes just obtained for $p\equiv1\pmod4$ are not identified
with each other: they come from the subgroups $\{e\}$ and $C_2$, respectively,
and these subgroups are not conjugate in $C_2$. By
Theorem~\ref{thm:main_classification}, they therefore define distinct
equivalence classes. Thus $\Rep(H_p)\simeq\cC_p^{C_2}$ has one equivalence
class of fiber functors if $p\equiv3\pmod4$, and two if
$p\equiv1\pmod4$.
\end{proof}
\section{Fiber functors on gaugings}\label{sec:gauging}

We now apply the classification to gaugings. Let $\cB$ be a braided fusion category and let
\[
\cE=\bigoplus_{g\in G}\cE_g
\]
be a faithful $G$-crossed braided fusion category with neutral component $\cE_e=\cB$. We denote by $\rho$ the crossed action of $G$ on $\cE$. Following the formulation of gauging in~\cite{CGPW16}, based on the extension theory of~\cite{ENO10}, the gauging determined by the $G$-crossed extension $\cE$ is the equivariantization
\[
\cE^G.
\]
The crossed braiding makes $\cE^G$ a braided fusion category, but the fiber-functor problem below uses only the underlying tensor category $\cE$ together with the action $\rho$.

\begin{definition}
For a subgroup $S\subseteq G$, set $\cE_S=\bigoplus_{s\in S}\cE_s$. A \emph{Clifford datum} for $\cE$ is a triple $(S,\cN,\widetilde{\cN})$, where $\cN$ is an indecomposable semisimple $\cB$-module category and $\widetilde{\cN}$ is an indecomposable semisimple $\cE_S$-module category on the same underlying finite semisimple category, whose restriction to $\cB$ is equivalent to the given $\cB$-module structure on $\cN$. We call $\widetilde{\cN}$ an $\cE_S$-extension of $\cN$.
\end{definition}

Given a Clifford datum $(S,\cN,\widetilde{\cN})$, put
\[
\Ind(S,\widetilde{\cN})
:=\Ind_{\cE_S}^{\cE}(\widetilde{\cN})
\simeq \cE\boxtimes_{\cE_S}\widetilde{\cN}.
\]
The result is an indecomposable $\cE$-module category. Conversely, by Clifford theory for graded fusion categories, every indecomposable $\cE$-module category is of this form.

Let $\mathfrak d=(S,\cN,\widetilde{\cN})$ be a Clifford datum and let $h\in G$. We define its $h$-twist by
\[
\mathfrak d^h=(h^{-1}Sh,\cN^h,\widetilde{\cN}^{\,h}),
\]
where $\widetilde{\cN}^{\,h}$ is the same underlying category as $\widetilde{\cN}$, regarded as an $\cE_{h^{-1}Sh}$-module category by
\[
X\overline{\otimes}^{\,h} N:=\rho_h(X)\overline{\otimes}N,
\qquad X\in \cE_{h^{-1}Sh},\; N\in\widetilde{\cN},
\]
and $\cN^h$ is its restriction to $\cB=\cE_e$. Then
\[
\Ind(S,\widetilde{\cN})^h\simeq
\Ind(h^{-1}Sh,\widetilde{\cN}^{\,h})
\]
as $\cE$-module categories.

\begin{definition}
Let $\mathfrak d=(S,\cN,\widetilde{\cN})$ be a Clifford datum. An \emph{$H$-equivariant Clifford lift} of $\mathfrak d$ is a choice, for every $h\in H$, of an equivalence of induced $\cE$-module categories
\[
\Phi_h:\Ind(S,\widetilde{\cN})\xrightarrow{\sim}
\Ind(h^{-1}Sh,\widetilde{\cN}^{\,h}),
\]
together with coherence isomorphisms making the corresponding functors
\[
\Ind(S,\widetilde{\cN})\longrightarrow
\Ind(S,\widetilde{\cN})^h
\]
an $H$-equivariant structure in the sense of Section~\ref{sec:preliminaries}.
\end{definition}

\begin{remark}
For each fixed $h\in H$, the equivalence criterion for induced module categories over graded fusion categories~\cite[Definition~5.10 and Proposition~5.12]{MM20} tests the existence of $\Phi_h$ by asking for an element $a_h\in G$ such that
\[
S=a_h(h^{-1}Sh)a_h^{-1}
\]
and, writing
\[
\cE_{a_h(h^{-1}Sh)}:=\bigoplus_{u\in a_h(h^{-1}Sh)}\cE_u
\]
for the corresponding $(\cE_S,\cE_{h^{-1}Sh})$-bimodule category, an
equivalence of $\cE_S$-module categories
\[
\cE_{a_h(h^{-1}Sh)}
\boxtimes_{\cE_{h^{-1}Sh}}
\widetilde{\cN}^{\,h}
\simeq
\widetilde{\cN}.
\]
It remains to require these equivalences, as $h$ varies, to satisfy the equivariant coherence.
\end{remark}

\begin{theorem}\label{thm:gauging_clifford_fiber_functors}
Let $\cE=\bigoplus_{g\in G}\cE_g$ be a faithful $G$-crossed braided fusion category with neutral component $\cB=\cE_e$. Fiber functors on $\cE^G$ correspond to equivalence classes of tuples
\[
(H,S,\cN,\widetilde{\cN},\mathbf{u})
\]
satisfying:
\begin{enumerate}
\item $H$ and $S$ are subgroups of $G$;
\item $(S,\cN,\widetilde{\cN})$ is a Clifford datum for $\cE$;
\item $\mathbf{u}$ is an $H$-equivariant Clifford lift of $(S,\cN,\widetilde{\cN})$;
\item for $\cM=\Ind(S,\widetilde{\cN})$, the induced action of $H$ on $\Irr(\cM)$ is transitive;
\item for some, hence every, simple object $X\in\cM$, the stabilizer cocycle
\[
[\alpha_X]\in H^2(\St_H(X),\ku^\times)
\]
is nondegenerate.
\end{enumerate}

Two tuples are identified when their associated equivariant $\cE$-module categories are conjugate in the sense of Definition~\textup{\ref{def:conjugate_equivariant_module_categories}}.
\end{theorem}

\begin{proof}
Apply Theorem~\ref{thm:main_classification} to the crossed action
$\rho:G\to\AutT(\cE)$. By Clifford theory for graded fusion
categories~\cite{Gal12, MM20}, every indecomposable $\cE$-module category
has the form
\[
\cM\simeq\Ind_{\cE_S}^{\cE}(\widetilde{\cN})
\simeq\cE\boxtimes_{\cE_S}\widetilde{\cN}
\]
for a Clifford datum $(S,\cN,\widetilde{\cN})$. Transporting the
$H$-equivariant structure on $\cM$ gives $\mathbf u$.

Conversely, such a tuple produces an $H$-equivariant indecomposable exact
semisimple $\cE$-module category. Items~\textup{(4)--(5)} and
Theorem~\ref{thm:main_classification} give a fiber functor. The equivalence
relation is exactly the conjugacy relation of
Theorem~\ref{thm:conjugacy_equivariantization}.
\end{proof}

\begin{corollary}\label{cor:prime_cyclic_gauging}
Let $G=C_p=\langle \tau\rangle$ be cyclic of prime order, and let $\cE=\bigoplus_{g\in G}\cE_g$ be a faithful $G$-crossed braided fusion category with neutral component $\cB$. Then the data of Theorem~\ref{thm:gauging_clifford_fiber_functors} reduce to the following possibilities.

\begin{enumerate}
\item $H=\{e\}$, $S=G$, and $\widetilde{\cN}$ is a rank-one $\cE$-module category. These fiber functors are the $G$-orbits of fiber functors on $\cE$ under the crossed action.

\item $H=G$, and $\cM=\Ind(S,\widetilde{\cN})$ carries a $G$-equivariant Clifford lift whose generator acts on $\Irr(\cM)$ as a single $p$-cycle. This has two subcases:
\begin{enumerate}
\item $S=\{e\}$, $\widetilde{\cN}=\cN$, and
\[
\cM=\cE\boxtimes_{\cB}\cN
=\bigoplus_{g\in G}\cM_g,
\qquad
\cM_g:=\cE_g\boxtimes_{\cB}\cN,
\]
where every $\cM_g$ has rank one and the generator of $G$ acts on their $p$ simple objects as a single $p$-cycle.
\item $S=G$, and $\widetilde{\cN}$ is an $\cE$-module category with $p$ simple objects, equipped with a $G$-equivariant structure whose generator acts as a $p$-cycle.
\end{enumerate}
\end{enumerate}
\end{corollary}

\begin{proof}
The only subgroups of $C_p$ are $\{e\}$ and $G$, and $H^2(L,\ku^\times)=0$ for both. Hence Corollary~\ref{cor:cohomologically_trivial} applies: the stabilizer condition in Theorem~\ref{thm:gauging_clifford_fiber_functors} is equivalent to freeness of the $H$-action on $\Irr(\cM)$.

If $H=\{e\}$, freeness and transitivity mean that $\cM$ has one simple object. If $S=\{e\}$, the restriction of $\Ind(S,\widetilde{\cN})$ to $\cB$ has one nonzero component for each element of $G$; hence its rank is at least $p$. Thus rank one forces $S=G$.

If $H=G$, a free transitive action of the cyclic group of order $p$ is exactly a single $p$-cycle. The two alternatives for $S$ are the only possible ones. For $S=\{e\}$, Clifford theory gives a decomposition
\[
\cM=\bigoplus_{g\in G}\cM_g,
\qquad
\cM_g=\cE_g\boxtimes_{\cB}\cN,
\]
into non-zero indecomposable $\cB$-module subcategories; see~\cite[Proposition~4.1]{Gal12}. Consequently,
\[
\#\Irr(\cM)=\sum_{g\in G}\#\Irr(\cM_g)\ge p,
\]
and equality holds if and only if every $\cM_g$ has rank one. In particular, equality implies that $\cN=\cM_e$ has rank one. This argument does not require an invertible $\cB$-bimodule category to preserve rank. The case $S=G$ is immediate from $\cM=\widetilde{\cN}$.
\end{proof}

\begin{remark}
Two layers of data are not visible from the neutral-component module category
alone. First, \(\cN\) must extend from \(\cB\) to \(\cE_S\). In the
obstruction theory of~\cite{MeMu12}, if
\(\cN\simeq\Mod_{\cB}(A)\), this requires \(S\)-stability and a splitting of
\[
1\longrightarrow\Aut_{\cB}(\cN)
\longrightarrow\Lambda\longrightarrow S\longrightarrow1,
\]
where \(\Lambda\) consists of invertible \(A\)-\(A\)-bimodules supported in
\(\cE_S\). After a splitting is chosen, an obstruction in
\(H^3(S,\ku^\times)\) governs associative multiplication on the resulting
graded algebra \(\widetilde A=\bigoplus_{s\in S}A_s\). Second, the induced
\(\cE\)-module category must be equivalent to its \(h\)-twists for
\(h\in H\), with compatible coherence. Thus the practical inputs are the
\(\cB\)-module category, its \(\cE_S\)-extension, and the equivariant lift;
the rank-one test is then supplied by Theorem~\ref{thm:gauging_clifford_fiber_functors}.
\end{remark}

\subsection{First reductions for small radical gaugings}

We apply the preceding description to the small radical gaugings classified
in~\cite{GNradical}. Table~\ref{tab:small_radical_gaugings}, adapted from
\cite[Table~3]{GNradical}, records the non-pointed nondegenerate entries and
the conclusions proved below. We follow the notation and parameter conventions
of that reference. Here \(D_r\) has order \(2r\); the symbols \(e\), \(h\), and
\(N\) denote the elliptic, hyperbolic, and norm forms, respectively; and
\(p,r\) are distinct odd primes except in the final row.

\begingroup
\footnotesize
\setlength{\tabcolsep}{2.2pt}
\renewcommand{\arraystretch}{1.16}
\begin{longtable}{|>{\raggedright\arraybackslash}p{1.05cm}|>{\raggedright\arraybackslash}p{2.45cm}|>{\centering\arraybackslash}p{0.65cm}|>{\raggedright\arraybackslash}p{2.0cm}|>{\raggedright\arraybackslash}p{3.25cm}|>{\raggedright\arraybackslash}p{2.35cm}|}
\caption{Non-pointed small radical gaugings from~\cite[Table~3]{GNradical}.}
\label{tab:small_radical_gaugings}\\
\hline
\(\FPdim\) & mantle \(\cB\) & \(G\) & generator action & description & fiber functors \\
\hline
\endfirsthead
\hline
\(\FPdim\) & mantle \(\cB\) & \(G\) & generator action & description & fiber functors \\
\hline
\endhead
\(4\) & \(\mathcal I_\eta\) & \(1\) & -- & Ising & no \\
\hline
\(8\) & \(\mathcal I_\eta\boxtimes\cC(C_2,q_{\eta^4})\) & \(1\) & -- & Ising \(\boxtimes\) pointed & no \\
\hline
\(4r\) & \(\cC(C_r,q_\mu)\) & \(C_2\) & \(x\mapsto-x\) & metaplectic & no \\
\hline
\(16\) & \(\mathcal I_\eta\boxtimes\cC(C_2^2,h)\) & \(1\) & -- & Ising \(\boxtimes\) pointed & no \\
\hline
\(16\) & \(\mathcal I_\eta\) & \(C_2\) & trivial & Ising \(\boxtimes\) pointed & no \\
\hline
\(16\) & \(\cC(C_4,q_\xi)\) & \(C_2\) & \(x\mapsto-x\) & two Ising factors & no \\
\hline
\(16\) & \(\cC(C_2,q_i)^{\boxtimes2}\) & \(C_2\) & swap & two Ising factors & no \\
\hline
\(16\) & \(\cC(C_2^2,e)\) & \(C_2\) & swap & two Ising factors & no \\
\hline
\(16\) & \(\cC(C_2^2,h)\) & \(C_2\) & swap & two Ising factors & no \\
\hline
\(8r\) & \(\cC(C_{2r},q_{i\mu})\) & \(C_2\) & \(x\mapsto-x\) & metaplectic & no \\
\hline
\(4r^2\) & \(\VecCat_\ku\) & \(D_r\) & -- & \(\Rep(D^\omega(D_r))\), \(\omega|_{C_r}\ne1\) & no \\
\hline
\(4r^2\) & \(\cC(C_r^2,h)\) & \(C_2\) & half-turn & \(\Rep(D^\omega(D_r))\), \(\omega|_{C_r}=1\) & yes iff \([\omega]=0\) \\
\hline
\(4r^2\) & \(\cC(C_r^2,h)\) & \(C_2\) & signed swap & \(\mathcal Z(\TY(A,\chi,\tau))\) & no \\
\hline
\(4r^2\) & \(\cC(C_r^2,e)\) & \(C_2\) & half-turn & elliptic gauging & no \\
\hline
\(4r^2\) & \(\cC(C_r^2,e)\) & \(C_2\) & signed swap & elliptic gauging & no \\
\hline
\(4pr\) & \(\cC(C_{pr},q_{\zeta\mu})\) & \(C_2\) & half-turn & metaplectic & no \\
\hline
\(4pr\) & \(\cC(C_{pr},q_{\zeta\mu})\) & \(C_2\) & mixed signs & metaplectic \(\boxtimes\) pointed & no \\
\hline
\(p^2r^2\), \(p\mid r-1\) & \(\cC(C_r^2,h)\) & \(C_p\) & order-\(p\) rotation & \(\Rep(D^\omega(C_r\rtimes C_p))\) & yes for every \([\omega]\) \\
\hline
\(p^2r^2\), \(p\mid r+1\) & \(\cC(C_r^2,e)\) & \(C_p\) & order-\(p\) rotation & elliptic prime-cyclic gauging & no \\
\hline
\(36\) & \(\cC(C_2^2,e)\) & \(C_3\) & order-\(3\) rotation & elliptic gauging & no \\
\hline
\end{longtable}
\endgroup

We first exclude the rows whose non-integrality is immediate. If a fusion category admits a fiber functor, finite Tannaka--Krein reconstruction identifies it with the representation category of a semisimple Hopf algebra~\cite[Theorem~5.3.12]{EGNO15}; hence Frobenius--Perron dimensions of simple objects are ordinary vector-space dimensions, and therefore integers. This excludes the rows containing an Ising factor, the metaplectic rows, and the metaplectic-times-pointed rows.

\begin{lemma}\label{lem:gt_descends_from_equivariantization}
Let a finite group $G$ act on a fusion category $\cC$. If the equivariantization $\cC^G$ is group-theoretical, then $\cC$ is group-theoretical.
\end{lemma}
\begin{proof}
The forgetful tensor functor $\cC^G\to\cC$ is dominant, and group-theoretical fusion categories are stable under dominant tensor images~\cite[Proposition~8.44]{ENO05}.
\end{proof}

The same obstruction applies to the rows described as centers $\mathcal{Z}(\TY(A,\chi,\tau))$, where $A\simeq C_r$ and $r$ is an odd prime. Let $m$ be the non-invertible simple object of $\TY(A,\chi,\tau)$. Then $\FPdim(m)=\sqrt{|A|}=\sqrt r$. The forgetful tensor functor
\[
\mathcal{Z}(\TY(A,\chi,\tau))\to \TY(A,\chi,\tau)
\]
is dominant, and tensor functors between fusion categories preserve Frobenius--Perron dimensions. Hence $m$ appears as a direct summand of the image of some simple central object $Z$. Writing the image of $Z$ in $\TY(A,\chi,\tau)$ as a direct sum of invertible objects and copies of $m$, we get
\[
\FPdim(Z)=a+b\sqrt r
\]
with $a,b\in\mathbb{Z}_{\geq0}$ and $b>0$. This number is not an integer. Thus $\mathcal{Z}(\TY(A,\chi,\tau))$ is itself non-integral.

\begin{proposition}\label{prop:twisted_double_test}
Let $K$ be a finite group. Let $\omega\in Z^3(K,\ku^\times)$, and choose the group-theoretical presentation
\[
\mathcal{Z}(\Vect{K}^{\omega})
\simeq
(\Vect{K\times K}^{\,\widetilde{\omega}})^*_{\ModCat(\Delta K,\widetilde{\mu})},
\]
where $\Delta K\subset K\times K$ is the diagonal subgroup. Then fiber functors on $\mathcal{Z}(\Vect{K}^{\omega})$ are parameterized by conjugacy classes of pairs $(L,\nu)$, with $\nu$ understood up to multiplication by a $2$-coboundary, such that:
\begin{enumerate}
\item $L\subseteq K\times K$ is a subgroup with $\Delta K\,L=K\times K$;
\item $\nu\in C^2(L,\ku^\times)$ satisfies $d\nu=\widetilde{\omega}|_L$;
\item the cocycle obtained from $\widetilde{\mu}\nu^{-1}$ on $\Delta K\cap L$ is nondegenerate.
\end{enumerate}
In this model the cohomology class of $\widetilde{\omega}$ is $\pr_1^*[\omega]-\pr_2^*[\omega]$. In particular, if $L$ is a complement to $\Delta K$, the last condition is automatic.
\end{proposition}
This is the fiber-functor classification for group-theoretical fusion categories~\cite[Section~2]{Nik08}, applied to Ostrik's group-theoretical model of the center of a pointed fusion category~\cite[Theorem~3.2, Proposition~3.1, Corollary~3.1 and Corollary~3.2]{Os03}. Changing $\nu$ by a $2$-coboundary gives an equivalent module category and does not affect the nondegeneracy condition.

For \(\omega=1\), the complement \(K\times\{e\}\) gives the usual
forgetful fiber functor on \(\Rep(D(K))\). For non-trivial \(\omega\),
Proposition~\ref{prop:twisted_double_test} reduces the question to the
restriction of \(\widetilde\omega\) to subgroups transverse to \(\Delta K\).

For the dihedral entries, write
\[
D_r=\langle\rho,s\mid \rho^r=s^2=e,\ s\rho s=\rho^{-1}\rangle,
\qquad R=\langle\rho\rangle,
\]
with \(r\) an odd prime. The representatives in
\cite[Section~6.3, Equation~(6.20)]{CGR00} identify
\(H^3(D_r,\ku^\times)\cong\mathbb Z/2r\mathbb Z\); denote the class indexed
by \(j\) by \([\alpha_j]\). Its restriction to \(R\) has parameter
\(-j\bmod r\), while its restriction to any reflection subgroup
\(C\cong C_2\) is the class \(j\bmod2\). Hence the row
\(\omega|_{C_r}=1\) consists of \(j=0,r\), and the other row consists of
\(j\not\equiv0\pmod r\).

\begin{lemma}\label{lem:dihedral_goursat}
Let $G=D_r$, with $r$ an odd prime, and put $R=\langle \rho\rangle$. If $L\subseteq G\times G$ is a subgroup such that $\Delta G\,L=G\times G$, then either $L=G\times H$ or $L=H\times G$ for some subgroup $H\subseteq G$, or else
\[
L=R\times C
\qquad\text{or}\qquad
L=C\times R
\]
for a reflection subgroup $C\subseteq G$.
\end{lemma}

\begin{proof}
The condition $\Delta G\,L=G\times G$ is equivalent to
\[
G=\{a^{-1}b\mid (a,b)\in L\}.
\]
Let $A=\pr_1(L)$ and $B=\pr_2(L)$. Then $AB=G$. Since $r$ is prime, the subgroups of $G$ are $1$, $R$, the reflection subgroups, and $G$. Hence either one of $A,B$ is $G$, or $\{A,B\}=\{R,C\}$ for a reflection subgroup $C$.

In the second case Goursat's lemma gives only the direct product, since $R$ and $C$ have no non-trivial common quotient. Thus $L=R\times C$ or $L=C\times R$.

Assume now that $A=G$. If the quotient appearing in Goursat's lemma is trivial, then $L=G\times B$. Suppose that this quotient is non-trivial. Since the normal subgroups of $G$ are $\{e\}$, $R$, and $G$, the only non-trivial quotients of $G$ are $G/R\cong C_2$ and $G$. In the quotient $G/R$, the relation defining $L$ gives $aR=bR$ for every $(a,b)\in L$. In the quotient $G$ case, necessarily $B=G$ and the relation is the graph of an automorphism of $G$; every automorphism of $D_r$ preserves $R$ and induces the identity on $G/R$, whence again $aR=bR$. Hence $a^{-1}b\in R$ for every $(a,b)\in L$, contradicting the displayed equality above. Thus only $L=G\times B$ can occur. The case $B=G$ is symmetric.
\end{proof}

\begin{proposition}\label{prop:dihedral_twisted_doubles_no_fiber}
Let $G=D_r$, with $r$ an odd prime, and let $\omega=\alpha_j$ be one of the cocycles above. Then $\mathcal{Z}(\Vect{G}^{\omega})\simeq \Rep(D^\omega(G))$ has a fiber functor if and only if $j=0$.
\end{proposition}

\begin{proof}
For $j=0$ this is the ordinary double, and the usual forgetful functor is a fiber functor.

Assume $j\ne0$ and suppose that a fiber functor exists. By Proposition~\ref{prop:twisted_double_test}, there is a subgroup $L\subseteq G\times G$ with $\Delta G\,L=G\times G$ such that $\widetilde{\omega}|_L$ is a coboundary. By Lemma~\ref{lem:dihedral_goursat}, either $L=G\times H$ or $L=H\times G$, or $L=R\times C$ or $L=C\times R$ with $C$ a reflection subgroup.

In the first two cases, $L$ contains $G\times\{e\}$ or $\{e\}\times G$. Since $[\widetilde{\omega}]=\pr_1^*[\omega]-\pr_2^*[\omega]$, the restriction of $\widetilde{\omega}$ to one of these subgroups is $\omega$ or $\omega^{-1}$, which is non-trivial because $j\ne0$. This is impossible.

It remains to consider $R\times C$ and $C\times R$. If $j\not\equiv0\pmod r$, the restriction of $\omega$ to $R$ is non-trivial; hence $\widetilde{\omega}|_L$ cannot be a coboundary. If $j\equiv0\pmod r$ and $j\ne0$, then $j=r$; since $r$ is odd, the restriction of $\omega$ to $C\cong C_2$ is the non-trivial class $(-1)^{ABC}$. Again $\widetilde{\omega}|_L$ cannot be a coboundary. Thus no non-trivial class $j$ gives a fiber functor.
\end{proof}

\begin{lemma}\label{lem:crossed_component_brpic}
Let $\cB$ be a nondegenerate braided fusion category, and let
\[
\cE=\bigoplus_{g\in G}\cE_g
\]
be a faithfully graded braided $G$-crossed extension of $\cB$.  Write
$T_g\in\Aut^{\mathrm{br}}_\otimes(\cB)$ for the crossed action of $g$.  Let
\[
\Phi:\operatorname{BrPic}(\cB)
   \xrightarrow{\ \simeq\ }
   \operatorname{EqBr}(\mathcal Z(\cB))
\]
be the canonical equivalence and let
\[
\iota:\cB\boxtimes\cB^{\mathrm{rev}}
   \xrightarrow{\ \simeq\ }
   \mathcal Z(\cB)
\]
be the braided equivalence determined by the braiding of $\cB$.
Then the invertible $\cB$-bimodule category $\cE_g$ satisfies
\begin{equation}\label{eq:crossed_component_brpic_formula}
\iota^{-1}\Phi([\cE_g])\iota
\simeq
\id_{\cB}\boxtimes T_g.
\end{equation}
\end{lemma}

\begin{proof}
By~\cite[Theorem~7.12]{ENO10}, the braided $G$-crossed structure determines a
morphism of categorical $2$-groups
\[
G\longrightarrow\operatorname{Pic}(\cB),
\qquad
 g\longmapsto[\cE_g].
\]
Recall that the alpha-induction homomorphism
\[
\partial:\operatorname{Pic}(\cB)
   \longrightarrow\Aut^{\mathrm{br}}_\otimes(\cB)
\]
is characterized by
$\alpha^-_{\cM}\circ\partial_{\cM}\simeq\alpha^+_{\cM}$ for every
invertible $\cB$-module category $\cM$~\cite[Section~5.4]{ENO10}.
For $X\in\cE_g$ and $Y\in\cB$, the composite of crossed braidings
\[
Y\otimes X
\xrightarrow{\ c_{Y,X}\ }
X\otimes Y
\xrightarrow{\ c_{X,Y}\ }
T_g(Y)\otimes X
\]
is a module-natural isomorphism
$\alpha^+_{\cE_g}(Y)\simeq\alpha^-_{\cE_g}(T_g(Y))$; the crossed hexagon
identities give the module-functor compatibility.  Hence
\[
\partial_{\cE_g}=T_g.
\]
Equivalently, this is the relation between the categorical $2$-group morphism
in~\cite[Theorem~7.12]{ENO10} and the structural map of the Picard crossed
module.

Under the Brauer--Picard equivalence, the image of
$\operatorname{Pic}(\cB)$ consists of the braided autoequivalences of
$\mathcal Z(\cB)$ that are trivializable on the first copy of $\cB$, while the
restriction to the second copy $\cB^{\mathrm{rev}}$ is the map $\partial$;
see~\cite[Lemma~4.4 and Corollary~4.9]{DN13}.  Since $\cB$ is
nondegenerate, the inverse construction in~\cite[Theorem~5.2]{ENO10} sends a
braided autoequivalence $T$ of $\cB$ to the bimodule category whose center
autoequivalence is $\id_{\cB}\boxtimes T$.  Taking $T=T_g$ proves~\eqref{eq:crossed_component_brpic_formula}.

Finally, in the extension classification of
\cite[Theorems~1.3 and~8.9]{ENO10}, the homomorphism
\[
c:G\longrightarrow\operatorname{BrPic}(\cB),
\qquad c(g)=[\cE_g],
\]
is fixed before the final $H^3(G,\ku^\times)$-parameter is chosen.  The latter
changes the associativity coherence but not the classes of the homogeneous
components.  This is also the convention in~\cite[Definition~2.1]{GPR24},
where the crossed action is fixed first and the parameter $\alpha$ labels the
remaining $H^3$-choices.
\end{proof}

\begin{lemma}\label{lem:pointed_defect_sector}
Let $A$ be a finite abelian group of odd order, let $\cB=\cC(A,q)$, and let
\[
\cE=\bigoplus_{g\in G}\cE_g
\]
be a braided $G$-crossed extension of $\cB$. Fix $g\in G$, and let
$\theta\in O(A,q)$ be the automorphism induced by $g_*$ on
$A=\Inv(\cB)$. Then, for every simple object $X\in\cE_g$, the following hold:
\begin{enumerate}
\item $A$ acts transitively on $\Irr(\cE_g)$ by left tensor product.
\item The stabilizer of $X$ is
\[
\St_A(X):=\{a\in A\mid a\otimes X\simeq X\}=(1-\theta)A.
\]
\item The assignment
\[
a+(1-\theta)A\longmapsto a\otimes X
\]
defines a bijection
\[
A/(1-\theta)A \xrightarrow{\ \simeq\ } \Irr(\cE_g).
\]
In particular,
\[
\#\Irr(\cE_g)=|\ker(1-\theta)|.
\]
\item The Frobenius--Perron dimension of $X$ is
\[
\FPdim(X)=\sqrt{|(1-\theta)A|}.
\]
\end{enumerate}
\end{lemma}

\begin{proof}
By~\cite[Theorem~6.1]{ENO10}, $\cE_g$ is an invertible, hence
indecomposable, $\cB$-module category. Since the simple objects of $\cB$
are the elements of $A$, tensoring by them permutes $\Irr(\cE_g)$. Each
$A$-orbit spans a $\cB$-module direct summand, so indecomposability implies
that this action is transitive.

The crossed braiding gives, for $a\in A$ and $X\in\cE_g$, isomorphisms
\[
a\otimes X\xrightarrow{\ c_{a,X}\ }X\otimes a
\xrightarrow{\ c_{X,a}\ }\theta(a)\otimes X.
\]
Tensoring on the left by $\theta(a)^{-1}$ shows that
$(a-\theta(a))\otimes X\simeq X$. Hence
\[
(1-\theta)A\subseteq\St_A(X).
\]

It remains to compute the number of simple objects. Write
\[
b_q(a,b)=q(a+b)q(a)^{-1}q(b)^{-1}
\]
for the symmetric bicharacter associated with $q$. Since $|A|$ is odd,
doubling is an automorphism of $A$; put
\[
c(a,b):=b_q(a/2,b).
\]
Then $c$ is a nondegenerate symmetric bicharacter, $c(a,a)=q(a)$, and
$\theta$ preserves $c$. Thus we may use the braided model in which the
underlying fusion category of $\cB$ is $\Vect{A}$ and the braiding is $c$.
Use $c$ to identify $\widehat A$ with $A$. Under the braided
equivalence
\[
\cB\boxtimes\cB^{\mathrm{rev}}\simeq\mathcal Z(\Vect{A}),
\qquad
(x,y)\longmapsto\bigl(x+y,c(x-y,-)\bigr),
\]
the autoequivalence $\id_{\cB}\boxtimes T_\theta$ is represented on
$A\oplus\widehat A\simeq A\oplus A$ by
\[
g_\theta=
\frac12
\begin{pmatrix}
1+\theta & 1-\theta\\
1-\theta & 1+\theta
\end{pmatrix},
\]
where $\frac12$ denotes the inverse of doubling on $A$. Indeed, the
coordinates $(u,v)=(x+y,x-y)$ are sent to
$(x+\theta(y),x-\theta(y))$. By Lemma~\ref{lem:crossed_component_brpic}, the invertible
$\Vect{A}$-bimodule category underlying $\cE_g$ is the one represented by
$\id_{\cB}\boxtimes T_\theta$, and hence by $g_\theta$ in these coordinates.
If $P:A\oplus\widehat A\to A$ is the first projection, then
\[
P\,g_\theta(0,v)=\frac12(1-\theta)v.
\]
Therefore~\cite[Proposition~10.8]{ENO10} gives
\[
\#\Irr(\cE_g)
=\bigl|\ker(P\,g_\theta|_{\widehat A})\bigr|
=|\ker(1-\theta)|
=|A/(1-\theta)A|.
\]
Transitivity now yields
\[
|A:\St_A(X)|=|\ker(1-\theta)|,
\]
so
\[
|\St_A(X)|=\frac{|A|}{|\ker(1-\theta)|}=|(1-\theta)A|.
\]
Together with the inclusion above, this proves
$\St_A(X)=(1-\theta)A$ and the asserted parametrization of simple
objects.

Finally, put $I=(1-\theta)A$. For $a\in A$, rigidity gives
\[
\Hom_{\cB}(a,X\otimes X^*)
\simeq
\Hom_{\cE_g}(a\otimes X,X).
\]
Since $X$ is simple, the latter space is one-dimensional for $a\in I$
and zero for $a\notin I$. Therefore
\[
X\otimes X^*\simeq\bigoplus_{a\in I}a,
\]
and hence
\[
\FPdim(X)^2=|I|=|(1-\theta)A|.
\]
\end{proof}

\begin{proposition}\label{prop:elliptic_four_r_squared_no_fiber}
Let $A=C_r^2$, with $r$ an odd prime. The two rows of Table~\ref{tab:small_radical_gaugings} with mantle $\cC(A,e)$ and $G=C_2$ admit no fiber functor.
\end{proposition}

\begin{proof}
We identify $A$ with a two-dimensional vector space over $\mathbb{F}_r$ and
let $\theta$ be the involution defining the action. Lemma~\ref{lem:pointed_defect_sector}
applies to the non-trivial homogeneous component.

If $\theta$ is one of the non-central involutions in the row
$(x,y)\mapsto(\pm y,\pm x)$, then $1-\theta$ has rank one over
$\mathbb{F}_r$; hence $(1-\theta)A$ has order $r$. The defect sector
contains simple objects of Frobenius--Perron dimension $\sqrt r$. A simple
object in the equivariantization lying over such a defect has dimension
$[C_2:\St_{C_2}(X)]\dim(\pi)\sqrt r$, where $\pi$ is the corresponding
irreducible projective representation of the stabilizer; this is the
orbit--stabilizer construction for simple equivariant objects recalled in
\cite[Subsection~7.5.1]{GP17}. Thus its dimension
is a positive integer multiple of $\sqrt r$, and the integrality obstruction
excludes this row.

It remains to consider the half-turn $\theta=-\id_A$. In this case
$(1-\theta)A=A$, so the non-trivial component has a unique simple object
$m$ of dimension $r$. Since the grading group has order two, $m^*$ lies in
the same component and hence $m^*\simeq m$. Also, every $a\in A$ fixes
$m$. Frobenius reciprocity therefore gives
\[
[m\otimes m^*:a]
=\dim\Hom(a\otimes m,m)=1
\qquad (a\in A).
\]
Thus
\[
m\otimes m\simeq\bigoplus_{a\in A}a,
\]
so the crossed extension has the Tambara--Yamagami fusion rules and is a
category $\TY(A,\chi,\tau)$. Since the row has elliptic metric group
$\cC(A,e)$, the associated symmetric form has no Lagrangian; equivalently,
the bilinear form $\chi$ in this Tambara--Yamagami category is elliptic.
We apply Theorem~\ref{thm:main_classification} to the equivariantization
by the induced action $T_{-\id_A}$.

The subgroup $H=\{e\}$ would give a fiber functor on $\TY(A,\chi,\tau)$ itself, which is impossible by Tambara's criterion~\cite[Propositions~4.1 and~4.2]{T00}, since $\chi$ is elliptic. Thus a fiber functor would have to come from $H=C_2$. Since $H^2(C_2,\ku^\times)=0$, Corollary~\ref{cor:cohomologically_trivial} forces an equivariant rank-two $\TY(A,\chi,\tau)$-module category, indecomposable as a $\TY(A,\chi,\tau)$-module category, on which the non-trivial element acts freely on the two simple objects.

To exclude this case, use the classification of module categories over Tambara--Yamagami categories~\cite[Section~9]{MeMu12}. In the form used in~\cite[Lemma~4.3 and Remark~4.5]{Nik08}, such a module category restricts to $\Vect{A}$ as
\[
\ModCat(A,\mu)\oplus \ModCat(A,\mu'),
\]
where $\ModCat(A,\mu)$ is the pointed $\Vect{A}$-module category attached to $\mu$, $[\mu]\ne0$ in $H^2(A,\ku^\times)$, and $\mu'$ is determined by $\chi$ and $\mu$. By~\cite[Lemma~9.1]{MeMu12}, the $\sigma$-translate of $\ModCat(H,\psi)$ is $\ModCat(R^\perp,\widetilde\psi)$, where $R=\operatorname{Rad}(\psi)$, $\xi_\psi$ is the alternating bicharacter associated with $\psi$, and $\widetilde\psi(a,b)=\psi(t_b,t_a)$ for elements $t_a$ satisfying $\xi_\psi(t_a,-)=\chi(a,-)$. Specializing this formula to $H=A$ and $R=\{0\}$, and writing the alternating form $\Alt(\mu)$ as $\lambda J$ and $\chi$ as a symmetric matrix $S$, gives the matrix formula used in~\cite[Remark~4.5]{Nik08}:
\[
\Alt(\mu')=\frac{-\det(S)}{\lambda}J.
\]
Since cohomology classes in $H^2(A,\ku^\times)$ are determined by their alternating forms, $[\mu]=[\mu']$ would imply $\lambda^2=-\det(S)$. Thus $\ModCat(A,\mu)$ and $\ModCat(A,\mu')$ are equivalent only if $-\det(S)$ is a square in $\mathbb{F}_r$, which is exactly the condition that $\chi$ be hyperbolic over $\mathbb{F}_r$. Since $\chi$ is elliptic, the two $\Vect{A}$-summands are not equivalent.

The automorphism $-\id_A$ acts trivially on $H^2(A,\ku^\times)$. If a $C_2$-equivariant structure on the rank-two module category acted freely on simple objects, the corresponding module equivalence
\[
U:\cM\to\cM^{T_{-\id_A}}
\]
would interchange the two simple objects. After restriction to $\Vect{A}$, this would force an equivalence $\ModCat(A,\mu)\simeq \ModCat(A,\mu')$, contradicting the preceding paragraph. Hence no such equivariant structure can act transitively on the two simple objects. This contradicts Corollary~\ref{cor:cohomologically_trivial}, and the half-turn row with elliptic mantle has no fiber functor.
\end{proof}

\begin{proposition}\label{prop:frobenius_prime_cyclic_positive}
Let $p,r$ be odd primes with $p\mid r-1$, and let
\[
K=\langle x,y\mid x^r=y^p=e,\; yxy^{-1}=x^a\rangle
\]
where $a$ has order $p$ in $\mathbb{F}_r^\times$. Then $\mathcal{Z}(\Vect{K}^{\omega})\simeq \Rep(D^\omega(K))$ admits a fiber functor for every $\omega\in Z^3(K,\ku^\times)$.
\end{proposition}

\begin{proof}
Let $R=\langle x\rangle$ and let $\pi:K\to P=K/R\simeq C_p$ be the quotient map. Consider the Lyndon--Hochschild--Serre spectral sequence
\[
E_2^{s,t}=H^s(P,H^t(R,\ku^\times))\Rightarrow H^{s+t}(K,\ku^\times).
\]
Since $|R|$ and $|P|$ are coprime, and $H^t(R,\ku^\times)$ is finite $r$-primary for $t>0$, the mixed terms $E_2^{s,t}$ with $s,t>0$ vanish. Moreover $E_2^{0,3}=H^3(R,\ku^\times)^P=0$, because $P$ acts on $H^3(R,\ku^\times)\simeq\mathbb{F}_r$ by multiplication by $a^2\ne1$. Hence the only contribution in total degree $3$ is $E_2^{3,0}=H^3(P,\ku^\times)$, and inflation gives
\[
H^3(K,\ku^\times)\simeq \pi^*H^3(P,\ku^\times).
\]
Thus, after changing $\omega$ within its cohomology class, which does not change the monoidal equivalence class of $\Vect{K}^{\omega}$, we may assume $\omega=\pi^*\eta$ for some $\eta\in Z^3(P,\ku^\times)$.

Choose the subgroup $\langle y\rangle\subset K$ as a section of $P$, and let $\iota:P\to P$ be inversion. Put $f:K\to K$ for the homomorphism obtained by composing $\pi$, $\iota$, and this section, and set
\[
L=\{(k,f(k))\mid k\in K\}\subset K\times K.
\]
Then $L$ is a subgroup of order $|K|$. If $(k,f(k))$ lies in $\Delta K$, then $k=f(k)$ belongs to $\langle y\rangle$ and its image in $P$ is fixed by inversion. Since $p$ is odd, this forces $k=e$. Hence $L\cap \Delta K=\{e\}$, and therefore $\Delta K\,L=K\times K$.

By Proposition~\ref{prop:twisted_double_test}, it remains only to check that $\widetilde{\omega}|_L$ is cohomologically trivial. In cohomology,
\[
[\widetilde{\omega}]|_L
=(\pi\circ \pr_1|_L)^*[\eta]
-(\pi\circ \pr_2|_L)^*[\eta]
=\pi^*[\eta]-(\iota\circ\pi)^*[\eta].
\]
The inversion automorphism of $P\simeq C_p$ acts trivially on $H^3(P,\ku^\times)$: under the identification
$H^3(P,\ku^\times)\simeq H^4(P,\mathbb{Z})$, inversion acts by multiplication by $(-1)^2=1$. Hence the displayed class vanishes; therefore there is a 2-cochain $\nu$ on $L$ with $d\nu=\widetilde{\omega}|_L$. Since $L\cap\Delta K=\{e\}$, the nondegeneracy condition in Proposition~\ref{prop:twisted_double_test} is automatic. Thus $(L,\nu)$ gives a fiber functor.
\end{proof}

\begin{proposition}\label{prop:elliptic_prime_cyclic_no_fiber}
Let $p,q$ be odd primes with $p\mid q+1$. Let $A=\mathbb{F}_{q^2}$, viewed as a two-dimensional vector space over $\mathbb{F}_q$, and let $N:A\to\mathbb{F}_q$ be the norm form, viewed through a fixed non-trivial additive character as a $\ku^\times$-valued quadratic form. For $c\in\ker(N)$ of order $p$ and $\alpha\in H^3(C_p,\ku^\times)$, let
\[
\cE_\alpha=\VecCat_{(A,N)}^{p,\alpha}
\]
be the corresponding $C_p$-crossed extension of $\cC(A,N)$. Then the gauging
\[
\cD_\alpha=(\cE_\alpha)^{C_p}
\]
admits no fiber functor.
\end{proposition}

\begin{proof}
The fusion rules of $\cE_\alpha$ are recalled in~\cite[Proposition~4.1]{GPR24}. Its trivial component is the pointed category $\Vect{A}$, and the non-trivial homogeneous components contain one simple object each. We write these non-invertible simples as $X_i$, $i\in C_p^\times$, with
\[
a\otimes X_i=X_i\otimes a=X_i,\qquad
X_i\otimes X_j=
\begin{cases}
qX_{i+j},& i+j\ne0,\\
\displaystyle\bigoplus_{a\in A}a,& i+j=0.
\end{cases}
\]
Lemma~\ref{lem:gt_descends_from_equivariantization} shows that if $\cD_\alpha$ were group-theoretical, then its de-equivariantization $\cE_\alpha$ would be group-theoretical. Thus the non-group-theoreticality of $\cD_\alpha$ already follows from~\cite[Theorem~4.4]{GPR24}. The fiber-functor assertion is stronger: non-group-theoretical Hopf representation categories can still admit fiber functors. We therefore apply Corollary~\ref{cor:prime_cyclic_gauging}.

First suppose that the subgroup $H$ in Corollary~\ref{cor:prime_cyclic_gauging} is trivial. Then a fiber functor on $\cD_\alpha$ would come from a fiber functor on $\cE_\alpha$. By reconstruction, this would give $\cE_\alpha\simeq \Rep(H_0)$ for a semisimple Hopf algebra $H_0$ of dimension $pq^2$. By~\cite[Corollary~1.2]{JL09}, all semisimple Hopf algebras of dimension $pq^2$ have group-theoretical representation categories. This contradicts~\cite[Theorem~4.4]{GPR24}, where these elliptic categories $\cE_\alpha$ are shown to be non-group-theoretical.

It remains to exclude the case $H=C_p$. In the subcase $S=C_p$ of Corollary~\ref{cor:prime_cyclic_gauging}, the Clifford datum has an $\cE_\alpha$-module category $\widetilde{\cN}$ whose restriction to $\Vect{A}$ is an indecomposable module category $\cN$. The number of simple objects is intrinsic to the underlying abelian category; hence $\#\Irr(\widetilde{\cN})=\#\Irr(\cN)$. By the pointed module-category classification~\cite[Section~2]{Nik08}, such $\Vect{A}$-module categories are indexed by pairs $(L,[\mu])$, where $L\subseteq A$ and $[\mu]\in H^2(L,\ku^\times)$, and have rank $[A:L]$. Hence their ranks are $1$, $q$, or $q^2$, never $p$. Thus this subcase cannot produce a rank-$p$ module category.

Consider now the subcase $S=\{e\}$, and suppose that it produces a fiber functor. Put
\[
\cM=\cE_\alpha\boxtimes_{\Vect{A}}\cN,
\qquad
\cM_t:=\cE_{\alpha,t}\boxtimes_{\Vect{A}}\cN
\quad (t\in C_p),
\]
where $\cE_{\alpha,t}$ denotes the homogeneous component of degree $t$. By Corollary~\ref{cor:prime_cyclic_gauging}, $\cM$ has rank $p$. The $p$ subcategories $\cM_t$ are non-zero indecomposable $\Vect{A}$-module categories, so each one has rank one. Let $m_t$ be the unique simple object of $\cM_t$; in particular, $\cM_0=\cN$ and its unique simple object is $m_0$.

Let
\[
U:\cM\longrightarrow \cM^\tau
\]
be the module equivalence attached to a generator $\tau\in C_p$, and write $U(m_0)\simeq m_s$. Restricting $U$ to $\cN=\cM_0$ and using the convention $X\overline\otimes^{\,\tau}M=\rho_\tau(X)\overline\otimes M$ gives an equivalence of $\Vect{A}$-module categories
\begin{equation}\label{eq:elliptic_shift_invariance}
\cN^{\tau^{-1}}
\simeq \cM_s
\simeq \cE_{\alpha,s}\boxtimes_{\Vect{A}}\cN.
\end{equation}
Indeed, the module-functor constraint for $U$ reads
\[
U(a\overline\otimes M)\simeq \rho_\tau(a)\overline\otimes U(M),
\qquad a\in A,
\]
which is precisely $\Vect{A}$-linearity after twisting the source by $\tau^{-1}$.

By the classification of indecomposable module categories over $\Vect{A}$, the rank-one category $\cN$ is determined by a class $[\mu]\in H^2(A,\ku^\times)$. The restriction of $\rho_\tau$ to $A$ is multiplication by $c$~\cite[Definition~3.1 and the proof of Proposition~4.2]{GPR24}, so twisting by $\tau^{-1}$ pulls this class back along multiplication by $c^{-1}$. Since $\ku^\times$ is divisible,
\[
H^2(A,\ku^\times)\simeq \Hom(\wedge^2 A,\ku^\times),
\]
and multiplication by $c$ acts on $\wedge^2 A$ by
\[
\det_{\mathbb F_q}(c)=N(c)=1.
\]
Therefore $\cN^{\tau^{-1}}\simeq\cN$, and~\eqref{eq:elliptic_shift_invariance} yields
\[
\cE_{\alpha,s}\boxtimes_{\Vect{A}}\cN\simeq\cN.
\]
If $s\ne0$, then $s$ generates $C_p$. Iterating the relative tensor product and using the graded multiplication equivalences shows that the same equivalence holds for every homogeneous component. Moreover, because $A$ is abelian, every rank-one $\Vect{A}$-module category is pointed. Indeed, a simple $\Vect{A}$-module endofunctor has the identity as its underlying functor, and its module structure is a character of $A$; hence the simple objects of the dual category are indexed by $\widehat A$ and are invertible. The group-theoretical criterion~\cite[Theorem~16]{MeMu12} would then imply that $\cE_\alpha$ is group-theoretical, contradicting~\cite[Theorem~4.4]{GPR24}. Hence $s=0$ and $U(m_0)\simeq m_0$.

Finally, for each $i\in C_p^\times$, the non-zero object $X_i\overline\otimes m_0$ belongs to the rank-one component $\cM_i$. Thus
\[
X_i\overline\otimes m_0\simeq n_i m_i
\]
for some integer $n_i>0$. The crossed action preserves the grading and the non-trivial component of degree $i$ has the unique simple object $X_i$, so $\rho_\tau(X_i)\simeq X_i$. Applying the module-functor constraint gives
\[
n_iU(m_i)
\simeq U(X_i\overline\otimes m_0)
\simeq \rho_\tau(X_i)\overline\otimes U(m_0)
\simeq X_i\overline\otimes m_0
\simeq n_i m_i.
\]
Semisimplicity implies $U(m_i)\simeq m_i$ for every $i$. Hence the generator fixes all simple objects of $\cM$ and cannot act as a $p$-cycle. This rules out the last subcase and proves the proposition.
\end{proof}

\begin{lemma}\label{lem:f4_elliptic_extension_non_gt}
Let $A=\mathbb{F}_4$ and let $N:A\to\mathbb{F}_2$ be the norm form, composed with the non-trivial additive character of $\mathbb{F}_2$. For every $\alpha\in H^3(C_3,\ku^\times)$, the crossed extension
\[
\cE_\alpha=\VecCat_{(A,N)}^{3,\alpha}
\]
is non-group-theoretical.
\end{lemma}

\begin{proof}
Write $q(a)=(-1)^{N(a)}$ and $\cB=\cC(A,q)$, so that the underlying fusion category of $\cB$ is $\Vect{A}$. Thus
\[
q(0)=1,
\qquad
q(a)=-1\quad(a\ne0).
\]
Let $T_c$ be the braided autoequivalence induced by multiplication by an element $c\in\mathbb{F}_4^\times$ of order three. By~\cite[Definition~2.1 and Example~3.3(1)]{GPR24}, the category $\cE_\alpha$ is the braided $C_3$-crossed extension whose crossed action is generated by $T_c$.  There is no additional $H^2$-choice here: multiplication by $3$ is an automorphism of the $2$-group $A$, so $H^2(C_3,A)=0$ for this action.

By Lemma~\ref{lem:crossed_component_brpic}, under the braided equivalence
\[
\iota:\cB\boxtimes\cB^{\mathrm{rev}}
\xrightarrow{\ \simeq\ }
\mathcal Z(\Vect{A}),
\]
the Brauer--Picard class of the degree-one component satisfies
\begin{equation}\label{eq:f4_extension_brpic_component}
\iota^{-1}\Phi([\cE_{\alpha,1}])\iota
\simeq
\id_{\cB}\boxtimes T_c.
\end{equation}
This formula is independent of $\alpha$: the parameter $\alpha$ changes the final associativity coherence of the extension, but not the homomorphism
$C_3\to\operatorname{BrPic}(\Vect{A})$ represented by its homogeneous components.

The group-theoretical criterion for graded extensions~\cite[Theorem~16 and Remark~17]{MeMu12} now says that $\cE_\alpha$ is group-theoretical if and only if a Lagrangian subcategory of $\mathcal Z(\Vect{A})$ is stable under the autoequivalence in~\eqref{eq:f4_extension_brpic_component}.

The metric group of $\cB\boxtimes\cB^{\mathrm{rev}}$ is $A\oplus A$ with quadratic form
\[
Q(x,y)=q(x)q(y)^{-1}.
\]
A Lagrangian subgroup $L\subseteq A\oplus A$ has order four and satisfies $Q|_L=1$. If $(x,0)\in L$ with $x\ne0$, then $Q(x,0)=-1$, a contradiction; similarly, $L$ contains no non-zero element of the form $(0,y)$. Hence both coordinate projections $L\to A$ are injective. Since $|L|=|A|$, there is an automorphism $f\in\Aut(A)$ such that
\[
L=\Gamma_f:=\{(x,f(x)):x\in A\}.
\]
Conversely, every such graph is Lagrangian, because $q$ takes the same value on every non-zero element of $A$.

If $\Gamma_f$ were invariant under $\id_A\oplus T_c$, then
\[
(x,cf(x))\in\Gamma_f
\qquad(x\in A).
\]
Comparison of the first coordinates gives $cf(x)=f(x)$ for all $x$. Since $f$ is surjective, this would force $c=1$, contrary to the choice of $c$. Thus no Lagrangian subcategory is $C_3$-stable, and the criterion cited above proves that $\cE_\alpha$ is non-group-theoretical.
\end{proof}

\begin{proposition}\label{prop:dimension36_no_fiber}
Let $A=\mathbb{F}_4$, viewed as an additive group, and let $N:A\to\mathbb{F}_2$ be the norm form, viewed through a fixed non-trivial additive character as a $\ku^\times$-valued quadratic form. For every $\alpha\in H^3(C_3,\ku^\times)$, the dimension $36$ gauging
\[
\cD_\alpha=\left(\VecCat_{(A,N)}^{3,\alpha}\right)^{C_3}
\]
admits no fiber functor.
\end{proposition}

\begin{proof}
Put $\cE_\alpha=\VecCat_{(A,N)}^{3,\alpha}$. We examine the alternatives in Corollary~\ref{cor:prime_cyclic_gauging}.

First suppose that $H=C_3$. If $S=C_3$, the Clifford datum contains an $\cE_\alpha$-module category $\widetilde{\cN}$ whose restriction to $\Vect{A}$ is an indecomposable module category $\cN$. These categories have the same underlying abelian category and therefore the same rank. Indecomposable module categories over $\Vect{A}$ are parameterized by pairs $(L,[\mu])$, where $L\subseteq A$ and $[\mu]\in H^2(L,\ku^\times)$, and their rank is $[A:L]$~\cite[Section~2]{Nik08}. Since $A\simeq C_2^2$, the possible ranks are $1$, $2$, and $4$, never $3$. Hence this subcase is impossible.

Assume next that $S=\{e\}$. Then
\[
\cM=\cE_\alpha\boxtimes_{\Vect{A}}\cN
     =\bigoplus_{t\in C_3}\cM_t,
\qquad
\cM_t=\cE_{\alpha,t}\boxtimes_{\Vect{A}}\cN.
\]
If this datum produced a fiber functor, Corollary~\ref{cor:prime_cyclic_gauging} would force every $\cM_t$ to have rank one. In particular, $\cN=\cM_0$ has rank one. Such module categories are parameterized by
\[
H^2(A,\ku^\times)
\simeq \Hom(\wedge^2A,\ku^\times)
\simeq C_2,
\]
where the first isomorphism uses the divisibility of $\ku^\times$. Thus there are exactly two equivalence classes of rank-one $\Vect{A}$-module categories.

The homogeneous components of $\cE_\alpha$ act on equivalence classes of indecomposable $\Vect{A}$-module categories by
\[
[\cN]\longmapsto
[\cE_{\alpha,t}\boxtimes_{\Vect{A}}\cN].
\]
The three classes $[\cM_t]$ form the $C_3$-orbit of $[\cN]$, and all of them have rank one. This orbit is therefore contained in a set of two elements. Since a $C_3$-orbit has cardinality one or three, the orbit must have cardinality one. Consequently,
\[
\cE_{\alpha,t}\boxtimes_{\Vect{A}}\cN\simeq\cN
\qquad(t\in C_3),
\]
so $\cN$ is $C_3$-stable.

Moreover, every rank-one $\Vect{A}$-module category is pointed. Indeed, its underlying semisimple category has one simple object, and its simple $\Vect{A}$-module endofunctors are indexed by the characters of $A$; all of them are invertible. The criterion~\cite[Theorem~16]{MeMu12} would therefore make $\cE_\alpha$ group-theoretical, contradicting Lemma~\ref{lem:f4_elliptic_extension_non_gt}. Thus the subcase $S=\{e\}$ is also impossible.

It remains to consider $H=\{e\}$. Corollary~\ref{cor:prime_cyclic_gauging} then requires a fiber functor on $\cE_\alpha$ itself. By reconstruction, this would give a semisimple Hopf algebra $H_1$ of dimension $12$ with $\cE_\alpha\simeq\Rep(H_1)$. Every semisimple Hopf algebra of dimension $pq^2$, and in particular of dimension $3\cdot2^2$, has a group-theoretical representation category~\cite[Corollary~1.2]{JL09}. This again contradicts Lemma~\ref{lem:f4_elliptic_extension_non_gt}. Therefore $\cD_\alpha$ admits no fiber functor.
\end{proof}

\begin{corollary}\label{cor:factorizable_hopf_realizations}
Among the non-pointed nondegenerate braided fusion categories in
Table~\ref{tab:small_radical_gaugings}, the categories admitting a
realization as \(\Rep(H)\) for a finite-dimensional semisimple factorizable
Hopf algebra \(H\) are exactly:
\begin{enumerate}
\item the ordinary dihedral doubles \(\Rep(D(D_r))\);
\item the hyperbolic prime-cyclic twisted doubles
\(\Rep(D^\omega(C_r\rtimes C_p))\), with \(p\mid r-1\), for every
\([\omega]\in H^3(C_r\rtimes C_p,\ku^\times)\).
\end{enumerate}
All other non-pointed entries in the table admit no fiber functor.
\end{corollary}

\begin{proof}
The fiber-functor conclusions are assembled in
Table~\ref{tab:small_radical_gaugings} from the integrality obstruction and
Propositions~\ref{prop:dihedral_twisted_doubles_no_fiber},
\ref{prop:elliptic_four_r_squared_no_fiber},
\ref{prop:frobenius_prime_cyclic_positive},
\ref{prop:elliptic_prime_cyclic_no_fiber}, and
\ref{prop:dimension36_no_fiber}. Finite Tannaka--Krein reconstruction turns
fiber functors into semisimple Hopf realizations~\cite[Theorem~5.3.12]{EGNO15};
for a nondegenerate braided representation category, the corresponding
quasitriangular Hopf algebra is factorizable~\cite[Chapter~VIII]{Kas95}.
\end{proof}

\bibliographystyle{alpha}
\bibliography{referencias}

@book{EGNO15,
  author    = {Pavel Etingof and Shlomo Gelaki and Dmitri Nikshych and Victor Ostrik},
  title     = {Tensor categories},
  series    = {Mathematical Surveys and Monographs},
  volume    = {205},
  publisher = {American Mathematical Society},
  address   = {Providence, RI},
  year      = {2015},
  pages     = {xvi+343},
  isbn      = {978-1-4704-2024-6},
  doi       = {10.1090/surv/205},
  url       = {https://doi.org/10.1090/surv/205}
}

@book{Saavedra72,
  author    = {Saavedra Rivano, Neantro},
  title     = {Cat{\'e}gories tannakiennes},
  series    = {Lecture Notes in Mathematics},
  volume    = {265},
  publisher = {Springer-Verlag},
  address   = {Berlin-New York},
  year      = {1972},
  pages     = {ii+418},
  isbn      = {978-3-540-05844-1},
  doi       = {10.1007/BFb0059108},
  url       = {https://doi.org/10.1007/BFb0059108}
}

@incollection{DM82,
  author    = {Deligne, Pierre and Milne, James S.},
  title     = {Tannakian categories},
  booktitle = {Hodge cycles, motives, and {S}himura varieties},
  series    = {Lecture Notes in Mathematics},
  volume    = {900},
  publisher = {Springer-Verlag},
  address   = {Berlin-New York},
  year      = {1982},
  pages     = {101--228},
  doi       = {10.1007/978-3-540-38955-2_4},
  url       = {https://doi.org/10.1007/978-3-540-38955-2_4}
}

@incollection{JS91,
  author    = {Joyal, Andr{\'e} and Street, Ross},
  title     = {An introduction to {T}annaka duality and quantum groups},
  booktitle = {Category theory},
  series    = {Lecture Notes in Mathematics},
  volume    = {1488},
  publisher = {Springer-Verlag},
  address   = {Berlin},
  year      = {1991},
  pages     = {413--492},
  doi       = {10.1007/BFb0084235},
  url       = {https://doi.org/10.1007/BFb0084235}
}

@book{Schauenburg92,
  author    = {Schauenburg, Peter},
  title     = {Tannaka Duality for Arbitrary {H}opf Algebras},
  series    = {Algebra-Berichte},
  number    = {66},
  publisher = {R. Fischer},
  address   = {M{\"u}nchen},
  year      = {1992},
  pages     = {57},
  isbn      = {978-3-88927-100-6}
}

@article{BV07,
  author  = {Brugui{\`e}res, Alain and Virelizier, Alexis},
  title   = {Hopf monads},
  journal = {Adv. Math.},
  fjournal = {Advances in Mathematics},
  volume  = {215},
  number  = {2},
  pages   = {679--733},
  year    = {2007},
  doi     = {10.1016/j.aim.2007.04.011},
  url     = {https://doi.org/10.1016/j.aim.2007.04.011},
  eprint  = {math/0604180}
}

@book{Kas95,
  author    = {Kassel, Christian},
  title     = {Quantum Groups},
  series    = {Graduate Texts in Mathematics},
  volume    = {155},
  publisher = {Springer-Verlag},
  address   = {New York},
  year      = {1995},
  isbn      = {978-1-4612-6900-7},
  doi       = {10.1007/978-1-4612-0783-2},
  url       = {https://doi.org/10.1007/978-1-4612-0783-2}
}

@article {CM17,
    AUTHOR = {Cuadra, Juan and Meir, Ehud},
     TITLE = {Orders of {N}ikshych's {H}opf algebra},
   JOURNAL = {J. Noncommut. Geom.},
  FJOURNAL = {Journal of Noncommutative Geometry},
    VOLUME = {11},
      YEAR = {2017},
    NUMBER = {3},
     PAGES = {919--955},
      ISSN = {1661-6952,1661-6960},
   MRCLASS = {16T05 (14L15 16G30 18D30)},
  MRNUMBER = {3713009},
MRREVIEWER = {Iv\'an\ Ezequiel\ Angiono},
       DOI = {10.4171/JNCG/11-3-5},
       URL = {https://doi.org/10.4171/JNCG/11-3-5},
}

@article {TY98,
    AUTHOR = {Tambara, Daisuke and Yamagami, Shigeru},
     TITLE = {Tensor categories with fusion rules of self-duality for finite
              abelian groups},
   JOURNAL = {J. Algebra},
  FJOURNAL = {Journal of Algebra},
    VOLUME = {209},
      YEAR = {1998},
    NUMBER = {2},
     PAGES = {692--707},
      ISSN = {0021-8693,1090-266X},
   MRCLASS = {18D10 (20K01)},
  MRNUMBER = {1659954},
MRREVIEWER = {Marie\ Choda},
       DOI = {10.1006/jabr.1998.7558},
       URL = {https://doi.org/10.1006/jabr.1998.7558},
}

@article {T00,
    AUTHOR = {Tambara, D.},
     TITLE = {Representations of tensor categories with fusion rules of
              self-duality for abelian groups},
   JOURNAL = {Israel J. Math.},
  FJOURNAL = {Israel Journal of Mathematics},
    VOLUME = {118},
      YEAR = {2000},
     PAGES = {29--60},
      ISSN = {0021-2172,1565-8511},
   MRCLASS = {18D10 (16W30)},
  MRNUMBER = {1776075},
MRREVIEWER = {Graham\ J.\ Ellis},
       DOI = {10.1007/BF02803515},
       URL = {https://doi.org/10.1007/BF02803515},
}

@article{ENO05,
  author  = {Etingof, Pavel and Nikshych, Dmitri and Ostrik, Victor},
  title   = {On fusion categories},
  journal = {Ann. of Math. (2)},
  fjournal = {Annals of Mathematics},
  volume  = {162},
  number  = {2},
  pages   = {581--642},
  year    = {2005},
  doi     = {10.4007/annals.2005.162.581},
  url     = {https://doi.org/10.4007/annals.2005.162.581}
}

@article {ENO10,
    AUTHOR = {Etingof, Pavel and Nikshych, Dmitri and Ostrik, Victor},
     TITLE = {Fusion categories and homotopy theory},
   JOURNAL = {Quantum Topol.},
  FJOURNAL = {Quantum Topology},
    VOLUME = {1},
      YEAR = {2010},
    NUMBER = {3},
     PAGES = {209--273},
      ISSN = {1663-487X,1664-073X},
   MRCLASS = {18D10 (16T05 55P60)},
  MRNUMBER = {2677836},
       DOI = {10.4171/QT/6},
       URL = {https://doi.org/10.4171/QT/6},
      NOTE = {With an appendix by Ehud Meir},
}

@article {CGPW16,
    AUTHOR = {Cui, Shawn X. and Galindo, C\'esar and Plavnik, Julia Yael
              and Wang, Zhenghan},
     TITLE = {On gauging symmetry of modular categories},
   JOURNAL = {Comm. Math. Phys.},
  FJOURNAL = {Communications in Mathematical Physics},
    VOLUME = {348},
      YEAR = {2016},
    NUMBER = {3},
     PAGES = {1043--1064},
      ISSN = {0010-3616,1432-0916},
   MRCLASS = {18D10 (81R50)},
  MRNUMBER = {3555361},
       DOI = {10.1007/s00220-016-2633-8},
       URL = {https://doi.org/10.1007/s00220-016-2633-8},
}

@article {Nik08,
    AUTHOR = {Nikshych, Dmitri},
     TITLE = {Non-group-theoretical semisimple {H}opf algebras from group
              actions on fusion categories},
   JOURNAL = {Selecta Math. (N.S.)},
  FJOURNAL = {Selecta Mathematica. New Series},
    VOLUME = {14},
      YEAR = {2008},
    NUMBER = {1},
     PAGES = {145--161},
      ISSN = {1022-1824,1420-9020},
   MRCLASS = {16W30 (18D10)},
  MRNUMBER = {2480712},
MRREVIEWER = {Ram\'on\ Gonz\'alez Rodr\'iguez},
       DOI = {10.1007/s00029-008-0060-1},
       URL = {https://doi.org/10.1007/s00029-008-0060-1},
}

@article {GP17,
    AUTHOR = {Galindo, C\'esar and Plavnik, Julia Yael},
     TITLE = {Tensor functors between {M}orita duals of fusion categories},
   JOURNAL = {Lett. Math. Phys.},
  FJOURNAL = {Letters in Mathematical Physics},
    VOLUME = {107},
      YEAR = {2017},
    NUMBER = {3},
     PAGES = {553--590},
      ISSN = {0377-9017,1573-0530},
   MRCLASS = {18D10 (20J06 20J15)},
  MRNUMBER = {3606516},
MRREVIEWER = {Ehud\ Meir},
       DOI = {10.1007/s11005-016-0914-y},
       URL = {https://doi.org/10.1007/s11005-016-0914-y},
}

@article {T01,
    AUTHOR = {Tambara, Daisuke},
     TITLE = {Invariants and semi-direct products for finite group actions
              on tensor categories},
   JOURNAL = {J. Math. Soc. Japan},
  FJOURNAL = {Journal of the Mathematical Society of Japan},
    VOLUME = {53},
      YEAR = {2001},
    NUMBER = {2},
     PAGES = {429--456},
      ISSN = {0025-5645,1881-1167},
   MRCLASS = {18D10 (18D05)},
  MRNUMBER = {1815142},
MRREVIEWER = {Graham\ J.\ Ellis},
       DOI = {10.2969/jmsj/05320429},
       URL = {https://doi.org/10.2969/jmsj/05320429},
}

@article {Gal12,
    AUTHOR = {Galindo, C\'esar},
     TITLE = {Clifford theory for graded fusion categories},
   JOURNAL = {Israel J. Math.},
  FJOURNAL = {Israel Journal of Mathematics},
    VOLUME = {192},
      YEAR = {2012},
    NUMBER = {2},
     PAGES = {841--867},
      ISSN = {0021-2172,1565-8511},
   MRCLASS = {18D10 (20C05)},
  MRNUMBER = {3009745},
MRREVIEWER = {Alexander\ Zimmermann},
       DOI = {10.1007/s11856-012-0055-7},
       URL = {https://doi.org/10.1007/s11856-012-0055-7},
}

@article {GM11,
    AUTHOR = {Galindo, C\'esar and Mombelli, Mart\'in},
     TITLE = {Module categories over finite pointed tensor categories},
   JOURNAL = {Selecta Math. (N.S.)},
  FJOURNAL = {Selecta Mathematica. New Series},
    VOLUME = {18},
      YEAR = {2012},
    NUMBER = {2},
     PAGES = {357--389},
      ISSN = {1022-1824,1420-9020},
   MRCLASS = {16D90 (16T05 18D10 19D23)},
  MRNUMBER = {2927237},
MRREVIEWER = {Alessandro\ Ardizzoni},
       DOI = {10.1007/s00029-011-0067-x},
       URL = {https://doi.org/10.1007/s00029-011-0067-x},
}

@article {MM20,
    AUTHOR = {Mej\'ia Casta\~no, Adriana and Mombelli, Mart\'in},
     TITLE = {Equivalence classes of exact module categories over graded tensor categories},
   JOURNAL = {Comm. Algebra},
  FJOURNAL = {Communications in Algebra},
    VOLUME = {48},
      YEAR = {2020},
    NUMBER = {10},
     PAGES = {4102--4131},
      ISSN = {0092-7872,1532-4125},
       DOI = {10.1080/00927872.2020.1755679},
       URL = {https://doi.org/10.1080/00927872.2020.1755679},
}

@article {MeMu12,
    AUTHOR = {Meir, Ehud and Musicantov, Evgeny},
     TITLE = {Module categories over graded fusion categories},
   JOURNAL = {J. Pure Appl. Algebra},
  FJOURNAL = {Journal of Pure and Applied Algebra},
    VOLUME = {216},
      YEAR = {2012},
    NUMBER = {11},
     PAGES = {2449--2466},
      ISSN = {0022-4049,1873-1376},
   MRCLASS = {18D10},
  MRNUMBER = {2927176},
       DOI = {10.1016/j.jpaa.2012.03.014},
       URL = {https://doi.org/10.1016/j.jpaa.2012.03.014},
}

@article {ENO11,
    AUTHOR = {Etingof, Pavel and Nikshych, Dmitri and Ostrik, Victor},
     TITLE = {Weakly group-theoretical and solvable fusion categories},
   JOURNAL = {Adv. Math.},
  FJOURNAL = {Advances in Mathematics},
    VOLUME = {226},
      YEAR = {2011},
    NUMBER = {1},
     PAGES = {176--205},
      ISSN = {0001-8708,1090-2082},
   MRCLASS = {18D10 (16D90 16T05)},
  MRNUMBER = {2735754},
MRREVIEWER = {Rongmin\ Lu},
       DOI = {10.1016/j.aim.2010.06.009},
       URL = {https://doi.org/10.1016/j.aim.2010.06.009},
}

@article {Gal11,
    AUTHOR = {Galindo, C\'esar},
     TITLE = {Clifford theory for tensor categories},
   JOURNAL = {J. Lond. Math. Soc. (2)},
  FJOURNAL = {Journal of the London Mathematical Society. Second Series},
    VOLUME = {83},
      YEAR = {2011},
    NUMBER = {1},
     PAGES = {57--78},
      ISSN = {0024-6107,1469-7750},
   MRCLASS = {18D10 (16D90)},
  MRNUMBER = {2763944},
MRREVIEWER = {Alexander\ Zimmermann},
       DOI = {10.1112/jlms/jdq064},
       URL = {https://doi.org/10.1112/jlms/jdq064},
}

@article {Gal17,
    AUTHOR = {Galindo, C\'esar},
     TITLE = {Coherence for monoidal {$G$}-categories and braided
              {$G$}-crossed categories},
   JOURNAL = {J. Algebra},
  FJOURNAL = {Journal of Algebra},
    VOLUME = {487},
      YEAR = {2017},
     PAGES = {118--137},
      ISSN = {0021-8693,1090-266X},
   MRCLASS = {18D10},
  MRNUMBER = {3671186},
MRREVIEWER = {Brendan\ Fong},
       DOI = {10.1016/j.jalgebra.2017.05.027},
       URL = {https://doi.org/10.1016/j.jalgebra.2017.05.027},
}

@article {Os03,
    AUTHOR = {Ostrik, Victor},
     TITLE = {Module categories over the {D}rinfeld double of a finite
              group},
   JOURNAL = {Int. Math. Res. Not.},
  FJOURNAL = {International Mathematics Research Notices},
      YEAR = {2003},
    VOLUME = {2003},
    NUMBER = {27},
     PAGES = {1507--1520},
      ISSN = {1073-7928,1687-0247},
   MRCLASS = {18D10 (16D90 17B10 18E10)},
  MRNUMBER = {1976233},
MRREVIEWER = {Volodymyr\ V.\ Lyubashenko},
       DOI = {10.1155/S1073792803205079},
       URL = {https://doi.org/10.1155/S1073792803205079},
}

@article{EO04,
  author  = {Pavel Etingof and Victor Ostrik},
  title   = {Finite tensor categories},
  journal = {Mosc. Math. J.},
  fjournal = {Moscow Mathematical Journal},
  volume  = {4},
  number  = {3},
  pages   = {627--654},
  year    = {2004},
  doi     = {10.17323/1609-4514-2004-4-3-627-654},
  url     = {https://doi.org/10.17323/1609-4514-2004-4-3-627-654},
  note    = {arXiv:math/0301027}
}

@article{CGR00,
  author  = {Coste, Antoine and Gannon, Terry and Ruelle, Philippe},
  title   = {Finite group modular data},
  journal = {Nuclear Physics B},
  volume  = {581},
  number  = {3},
  pages   = {679--717},
  year    = {2000},
  doi     = {10.1016/S0550-3213(00)00285-6},
  url     = {https://doi.org/10.1016/S0550-3213(00)00285-6},
  eprint  = {hep-th/0001158}
}

@article{GPR24,
  author  = {Galindo, C{\'e}sar and Plavnik, Julia and Rowell, Eric C.},
  title   = {Integral non-group-theoretical modular categories of dimension {$p^2q^2$}},
  journal = {Bull. Belg. Math. Soc. Simon Stevin},
  volume  = {31},
  number  = {4},
  pages   = {516--525},
  year    = {2024},
  doi     = {10.36045/j.bbms.240415},
  url     = {https://doi.org/10.36045/j.bbms.240415},
  eprint  = {2404.03826}
}

@article{JL09,
  author  = {Jordan, David and Larson, Eric},
  title   = {On the classification of certain fusion categories},
  journal = {J. Noncommut. Geom.},
  volume  = {3},
  number  = {3},
  pages   = {481--499},
  year    = {2009},
  doi     = {10.4171/JNCG/44},
  url     = {https://doi.org/10.4171/JNCG/44},
  eprint  = {0812.1603}
}

@article{GNradical,
  author  = {Green, Jason and Nikshych, Dmitri},
  title   = {The {T}annakian radical and the mantle of a braided fusion category},
  journal = {Adv. Math.},
  fjournal = {Advances in Mathematics},
  volume  = {482},
  pages   = {110622},
  year    = {2025},
  doi     = {10.1016/j.aim.2025.110622},
  url     = {https://doi.org/10.1016/j.aim.2025.110622},
  eprint  = {2412.16722}
}

@article{DN13,
  author  = {Davydov, Alexei and Nikshych, Dmitri},
  title   = {The {P}icard crossed module of a braided tensor category},
  journal = {Algebr. Number Theory},
  volume  = {7},
  number  = {6},
  pages   = {1365--1403},
  year    = {2013},
  doi     = {10.2140/ant.2013.7.1365},
  url     = {https://doi.org/10.2140/ant.2013.7.1365},
  eprint  = {1202.0061}
}

\end{document}